\RequirePackage{ifthen}

\ifthenelse{\isundefined{\philtrans}}{
\documentclass[reqno,onefignum,final]{siamart171218}

\usepackage{graphicx}
\graphicspath{{figs/}}
\usepackage{amsmath,amssymb}
\usepackage{bm}
\usepackage[caption=false]{subfig}

\addtolength{\oddsidemargin}{.625in}
\addtolength{\evensidemargin}{.625in}

\bibliographystyle{siamplain}

\usepackage{lipsum}
\usepackage{amsfonts}
\usepackage{graphicx}
\usepackage{epstopdf}
\usepackage{algorithmic}
\ifpdf
  \DeclareGraphicsExtensions{.eps,.pdf,.png,.jpg}
\else
  \DeclareGraphicsExtensions{.eps}
\fi

\newcommand{\TheTitle}{Winding of a Brownian particle\\ around a point vortex}
\newcommand{\TheAuthors}{Huanyu Wen and Jean-Luc Thiffeault}

\headers{Winding of a Brownian particle around a point vortex}{\TheAuthors}

\title{{\TheTitle}}

\author{%
  Huanyu Wen%
  \thanks{Department of Mathematics, University of Wisconsin -- Madison,
    Madison, WI 53706, USA
  }
  \and
  Jean-Luc Thiffeault%
  \footnotemark[1]{}
  \thanks{\email{jeanluc@math.wisc.edu}}
}

\ifpdf
\hypersetup{
  pdftitle={\TheTitle},
  pdfauthor={\TheAuthors}
}
\fi

\newsiamremark{remark}{Remark}

}{

\documentclass{rstransa_jlt}

\usepackage{graphicx}
\graphicspath{{figs/}}
\usepackage{amsmath,amssymb}
\usepackage{bm}
\usepackage[caption=false]{subfig}
\usepackage[capitalise]{cleveref}

\bibliographystyle{rsta}

\title{Winding of a Brownian particle around a point vortex}

\author{%
  Huanyu Wen
  and
  Jean-Luc Thiffeault
}

\address{%
  Department of Mathematics \\
  University of Wisconsin -- Madison \\
  Madison, WI 53706, USA
}

\corres{Jean-Luc Thiffeault\\
\email{jeanluc@math.wisc.edu}}

\subject{%
  60J65, 
  76B47, 
  60H10  
}


\newtheorem{proposition}{\bf Proposition}[section]

}

\newcommand{\mathnotation}[2]{\newcommand{#1}{\ensuremath{#2}}}

\newcommand{\Order}[1]{\ensuremath{\mathcal{O}\!\l(#1\r)}}
\newcommand{\nofrac}[2]{#1/#2}

\renewcommand{\l}{\left}                
\renewcommand{\r}{\right}               
\mathnotation{\ee}{{\mathrm e}}         
\mathnotation{\imi}{\mathrm{i}}         
\mathnotation{\ldef}{\mathrel{\raisebox{.069ex}{:}\!\!=}}
\mathnotation{\dint}{\mathop{}\!\mathrm{d}}
\mathnotation{\dif}{\mathop{}\!\mathrm{d}}
\mathnotation{\E}{\mathbb{E}}           
\mathnotation{\RN}{\mathbb{R}}          
\mathnotation{\CN}{\mathbb{C}}          
\mathnotation{\p}{p}                    
\renewcommand{\P}{P}                    
\mathnotation{\W}{W}                    
\mathnotation{\X}{X}                    
\mathnotation{\x}{x}                    
\mathnotation{\R}{\rho}                 
\mathnotation{\kd}{k}                   
\mathnotation{\I}{\mathrm{I}}           
\mathnotation{\at}{{\tilde a}}         
\mathnotation{\Ltwo}{\mathrm{L}^2}      

\mathnotation{\convdist}{\xrightarrow{\enskip d\enskip}}

\mathnotation{\Rv}{\bm{R}}
\mathnotation{\rv}{\bm{r}}

\newcommand{\D}[2]{\frac{\partial #1}{\partial #2}}
\DeclareMathOperator{\sech}{sech}
\DeclareMathOperator{\sign}{sgn}
\DeclareMathOperator{\Gammadist}{Gamma}

\begin{document}

\ifthenelse{\isundefined{\philtrans}}
{

\maketitle

}
{
}

\begin{abstract}
  We derive the asymptotic winding law for a Brownian particle in the plane
  subjected to a tangential drift due to a point vortex.  For winding around a
  point, the normalized winding angle converges to an inverse Gamma
  distribution.  For winding around a disk, the angle converges to a
  distribution given by an elliptic theta function.  For winding in an
  annulus, the winding angle is asymptotically Gaussian with a linear drift
  term.  We validate our results with numerical simulations.
\end{abstract}



\ifthenelse{\isundefined{\philtrans}}{}
{

\begin{fmtext}

}

\section{Introduction}

Let~$\Rv(t)$ be a two-dimensional Brownian motion in $\RN^2$ with unit
diffusivity, so that $\E|\Rv(t)|^2 = 4t$.  Take $\Theta(t)$ to be the total
winding angle with respect to the origin accumulated by $\Rv(t)$ up to time
$t$.  It is a known fact that Brownian motion avoids the origin with
probability~$1$, so the winding angle is well-defined. Since the winding angle
takes value in $(-\infty,\infty)$ instead of $[0,2\pi)$, we can view the
Brownian motion as taking place in the universal cover of
$\CN\backslash\{0\}$, i.e., the Riemann surface of $\log z$.

Spitzer~\cite{Spitzer1958} in 1958 showed that the normalized winding angle
converges to a \textit{standard Cauchy} distribution as
$t\rightarrow \infty$,
\begin{equation}
  \frac{2\Theta(t)}{\log t}
  \convdist
  \X,
  \qquad
  \p_\X(\x) = \frac{1}{\pi} \frac{1}{1+\x^2}\,,
  \label{eq:Spitzer}
\end{equation}
where~$\p_\X(\x)$ denotes the probability density of~$\X$.  Spitzer proved
this by solving the Kolmogorov backward equation in a wedge region, then
taking the limit of infinite wedge angle.  A key feature of Spitzer's law is
that the winding angle has infinite variance, which is due to the roughness of
the Brownian trajectory generating large winding near the
origin~\cite{Rudnick1987}.  In fact, all positive integer moments diverge.
This divergence is undesirable when modeling physical problems, such as
flexible polymers, and there are several ways to regularize it.

\ifthenelse{\isundefined{\philtrans}}{}{

\end{fmtext}

\maketitle

}

Instead of a continuous process such as Brownian motion, let $\Rv(n)$ be a
random walk on a square lattice of $\RN^2$ that avoids the origin, with
i.i.d.\ steps~$\Delta\Rv$.  B\'{e}lisle~\cite{Belisle1989} showed that if
$\Delta\Rv$ is bounded or absolutely continuous with respect to Lebesgue
measure, then the winding angle follows a \emph{standard hyperbolic secant}
distribution,
\begin{equation}
  \frac{2\Theta(n)}{\log n}
  \convdist
  \X,
  \qquad
  \p_\X(\x) = \tfrac{1}{2} \sech(\pi\x/2),
  \label{eq:sech}
\end{equation}
as $n \rightarrow \infty$.  All moments now exist, since it is impossible for
$\Rv(n)$ to get infinitely close to the origin to generate large winding.  The
main idea behind the proof, introduced by Messulam \& Yor~\cite{Messulam1982},
is to consider two types of windings, the big windings (accumulated outside of
a certain disk) and the small windings (accumulated inside).  The same result
was obtained by Rudnick \& Hu~\cite{Rudnick1987} using a different method.
Berger~\cite{Berger1987} and Berger \& Roberts~\cite{Berger1988} studied a
walk with random angle but fixed step length and also observed convergence
to~\cref{eq:sech}.

Another way to regularize the original problem is to put a finite-sized
obstacle around the origin to prevent $\Rv(t)$ from entering its vicinity.
Specifically, let $\Rv(t)$ be a planar Brownian motion as above but with a
disk of radius $a$ carved out around the origin.  Assuming reflecting boundary
conditions on the disk's boundary, Grosberg \& Frisch~\cite{Grosberg2003}
showed that
\begin{equation}
  \frac{2\Theta(t)}{\log(4t/a^2)}
  \convdist
  \X,
  \qquad
  \p_\X(\x) = \tfrac{1}{2} \sech(\pi\x/2).
  \label{eq:GFsech}
\end{equation}
Revisiting Spitzer's case, they also gave a more precise normalization factor
for the winding angle, but we shall see in \cref{apx:revisit} that a small
correction is needed.

In the present paper we derive asymptotic winding laws for a Brownian particle
subjected to a point vortex at the origin.  The drift due to the point vortex
is in the tangential direction in the standard polar coordinates, and its
amplitude has magnitude~$\beta/r^2$, where~$\beta>0$ and~$r$ is the distance
from the origin.  This kind of vortical drift arises in inviscid fluid
motion~\cite{Lamb}.  Moreover, we shall see that this particular tangential
drift preserves eigensolutions of the separated Fokker--Planck equation in the
form of Bessel functions, albeit with complex argument, and is thus amenable
to analytical treatment.  We find the limit distribution for three natural
cases:

\paragraph{(i)} For the case where the particle winds around the origin
(\cref{fig:path_point,sec:pt}), we find convergence to an inverse
Gamma distribution, i.e.,
\begin{equation}
  \frac{\beta\log^2\!{\l(\nofrac{4t}{r_0^2\,\ee^\gamma}\r)}}{8\Theta(t)}
  \convdist
  \Gammadist(\tfrac{1}{2},\tfrac{1}{2}),
\end{equation}
where~$r_0$ is the initial radial coordinate, and $\gamma$ is the
Euler--Mascheroni constant.

\paragraph{(ii)} When the particle winds outside a disk of radius~$a$
(\cref{fig:path_inner_disk,sec:inner_disk}), the winding angle
distribution converges to the derivative~$\vartheta_2'$ of a second elliptic
theta function,
\begin{equation}
  \frac{4\Theta(t)}{\beta\,\log^2({4t}/{a^2\,\ee^{2\gamma}})}
  \convdist
  \X,
  \qquad
  \p_\X(\x)
  =
  -\tfrac{\pi}{2}\,\vartheta_2'
  \bigl(\tfrac{\pi}{2}\,,\,\ee^{-\pi^2 \x}\bigr)\,\chi_{(\x>0)}\,,
\end{equation}
where by convention the derivative in~$\vartheta_2'$ is with respect to the
first argument, and~$\chi$ is the indicator function.

\paragraph{(iii)} Finally, we show that for a particle winding inside an
annulus~$a < r < b$ (\cref{fig:path_annulus,sec:annulus}), the
winding angle distribution converges to a Gaussian with linear drift:
\begin{equation}
  \frac{\Theta(t)-A(t)\beta}{\sqrt{2A(t)}}
  \convdist
  N(0,1),
  \qquad
  A(t) = \frac{2 t}{b^2-a^2}\log(\nofrac{b}{a}).
\end{equation}

In all three cases we compare to numerical simulations to exhibit the
convergence to the limiting distribution.  Because the normalizers in the
first two cases are logarithmic in~$t$, it is crucial to have the
precise~$\Order{1}$ constants inside the~$\log$ when comparing to numerical
results, since otherwise astronomical values of~$t$ are required to see
convergence.

Existing results most closely related to ours either involve drift or bias the
Brownian motion in some way to promote winding.  Le Gall \&
Yor~\cite{LeGall1986} analyze a general planar Brownian motion with drift and
obtain a modification of Spitzer's law.  However, their drift must satisfy an
integrability condition that fails for a point vortex.  Comtet et
al.~\cite{Comtet1993b} studied diffusion with drift~$-\nabla \mathcal{U}(r)$
due to a radial potential in a disk of radius~$R$.  By picking the special
form $\mathcal{U} = (\alpha-1)\log(-\log(r/R_0))$, they found an
$\alpha$-stable L\'{e}vy distribution for the normalized winding angle
$\X(t) = \Theta(t)/t^{1/\alpha}$, with~$1 \leq \alpha^2 < 4$.  They found
convergence to a Gaussian winding angle for an annular domain.  Toby \&
Werner~\cite{Toby1995} showed that Brownian motion reflected at an angle along
an outer boundary also increases the winding.  Drossel \&
Kardar~\cite{Drossel1996} observed that chiral defects can promote winding.
Vakeroudis~\cite{Vakeroudis2015} considered winding of a complex
Ornstein--Uhlenbeck process.

\section{Brownian motion with tangential drift}

Consider the two-dimensional continuous-time stochastic process
$\Rv(t)=(R_1(t),R_2(t))$, starting at a point~$\Rv(0)=\rv_0 \ne \bm{0}$, and
obeying the SDE
\begin{equation}
  \dif\Rv
  =
  \Rv^\perp\,\Omega(\lvert\Rv\rvert,t) \dif t
  +
  \sqrt{2}\,\dif\bm{W}
  \label{eq:SDE}
\end{equation}
where~$\Rv^\perp = (-R_2,R_1)$ and~$\bm{W}(t)$ is a vector of two independent
standard Brownian motions.  The function~$\Omega(r,t)$ corresponds to a
tangential drift that spins particles around the origin.

Let $\p(\rv,t\,|\,\rv_0,0)$ be the transition probability density of $\Rv(t)$,
which satisfies the Fokker--Planck equation
\begin{equation}
  \D{\p}{t} + \Omega(r,t)\,\D{\p}{\theta} = \Delta \p,
  \qquad \p(\rv,0\,|\,\rv_0,0) = \delta(\rv-\rv_0).
  \label{eq:fund_pde}
\end{equation}
Our main goal is to study the case where $\Omega(r,t) = \nofrac{\beta}{r^2}$,
for a constant $\beta>0$, which corresponds to the flow field around a steady
two-dimensional point vortex~\cite{Lamb}.  Without loss of generality, we take
$\rv_0=(r_0,0)$ with~$r_0>0$.  Define in polar coordinates
\begin{equation}
  \P(r,\theta,t) = \p(\rv,t\,|\,\rv_0,0)
\end{equation}
where we suppressed the~$r_0$ dependence for simplicity.  \Cref{eq:fund_pde}
is now
\begin{equation}
  \D{\P}{t} + \frac{\beta}{r^2}\D{\P}{\theta}
  = \D{\P}{r} + \frac{1}{r}\D{\P}{r}
    +\frac{1}{r^2}\D{\P}{\theta},
  \qquad \P(r,\theta,0) = \frac{1}{r}\,\delta(r-r_0)\,\delta(\theta).
  \label{eq:fund_pde2}
\end{equation}
\Cref{eq:fund_pde2} has separated solutions of the form
\begin{equation}
  \ee^{-\lambda^2t}\,\ee^{\imi\mu\theta}\,\R(r)
\end{equation}
where $0 < \lambda \in \RN$, $\mu \in \RN$, and~$\R(r)$ obeys the eigenvalue
problem
\begin{equation}
  \R'' + \frac{1}{r}\R' + \biggl(\lambda^2
    - \frac{\kd_\mu^2}{r^2}\biggr)\R = 0,
  \qquad
  \kd_\mu = \sqrt{\mu^2+\imi\beta\mu}\,,
  \label{eq:char}
\end{equation}
where we take the branch for~$\kd_\mu$ with nonnegative real part.  Normally,
\cref{eq:fund_pde2} is solved with $2\pi$-periodic boundary conditions
in~$\theta$, so that~$\mu$ is a discrete eigenvalue.  However, to take the
winding into account we must allow for~$-\infty < \theta < \infty$, with the
boundary conditions that~$\P(r,\pm\infty,t)=0$.  From here, the solution to
\cref{eq:char} depends on boundary conditions in~$r$, which naturally divide
our discussion into three cases: winding around a point (\cref{sec:pt}),
around a disk (\cref{sec:inner_disk}), and inside an annulus
(\cref{sec:annulus}).

\begin{figure}
  \centering
  \subfloat[]{%
    \label{fig:path_point}%
    \includegraphics[width=.3\textwidth]{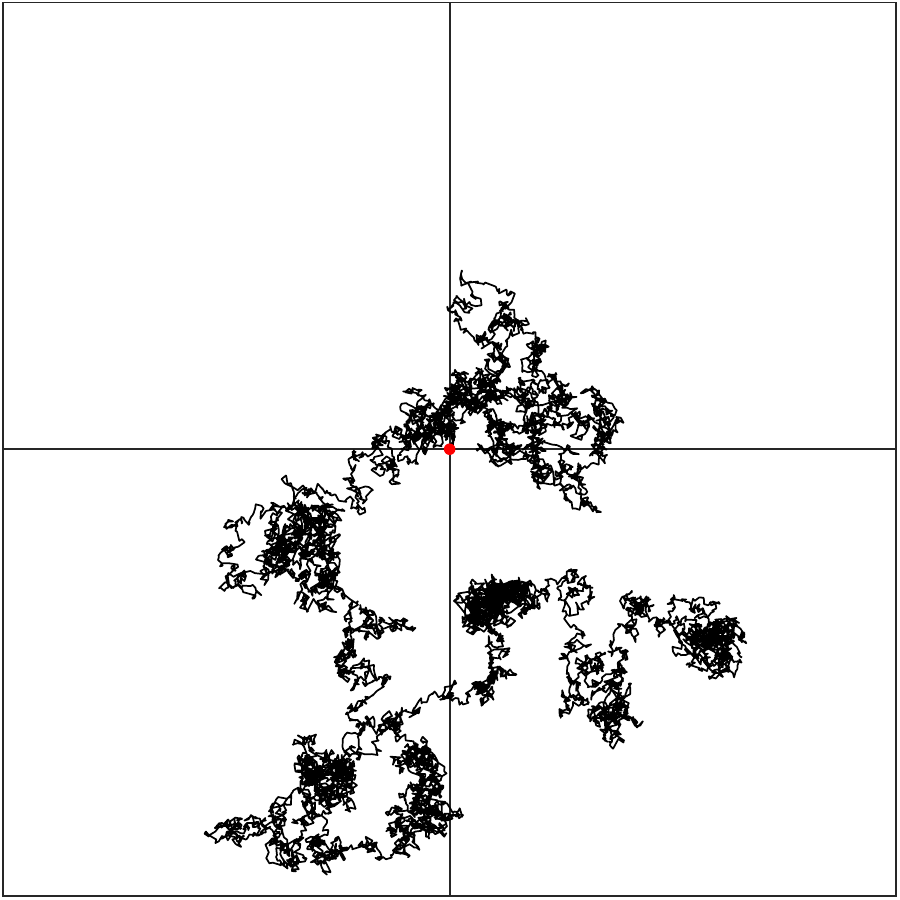}}
  \hspace{.01\textwidth}
  \subfloat[]{%
    \label{fig:path_inner_disk}%
    \includegraphics[width=.3\textwidth]{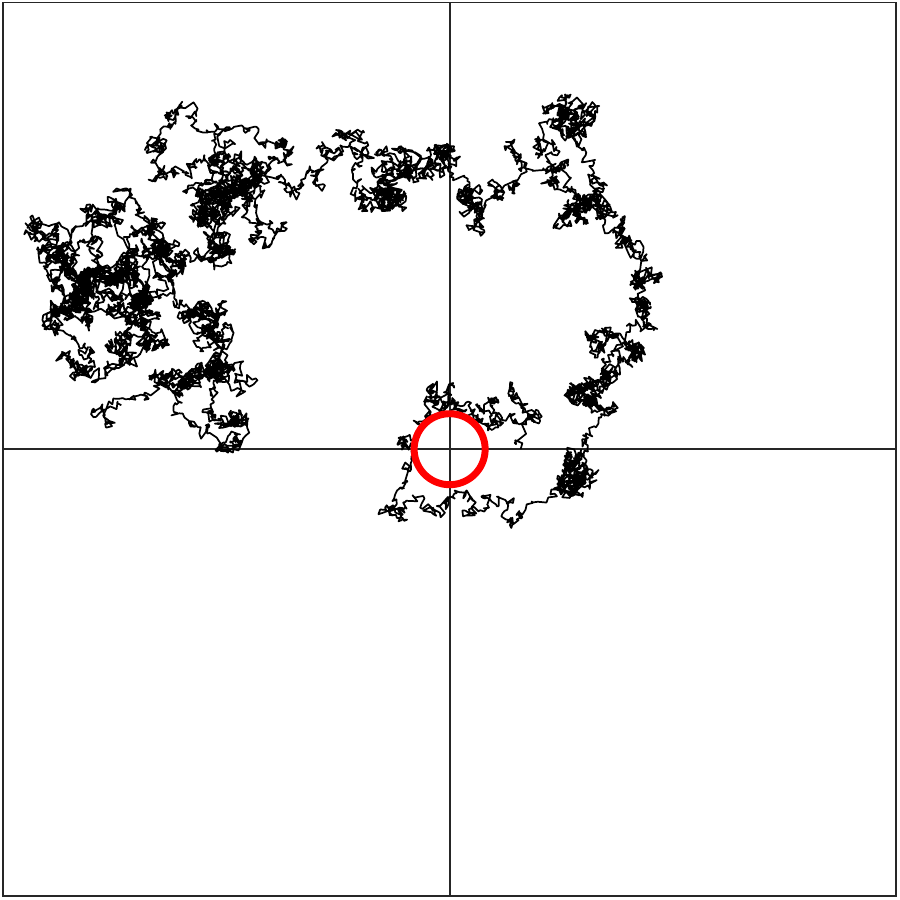}}
  \hspace{.01\textwidth}
  \subfloat[]{%
    \label{fig:path_annulus}%
    \includegraphics[width=.3\textwidth]{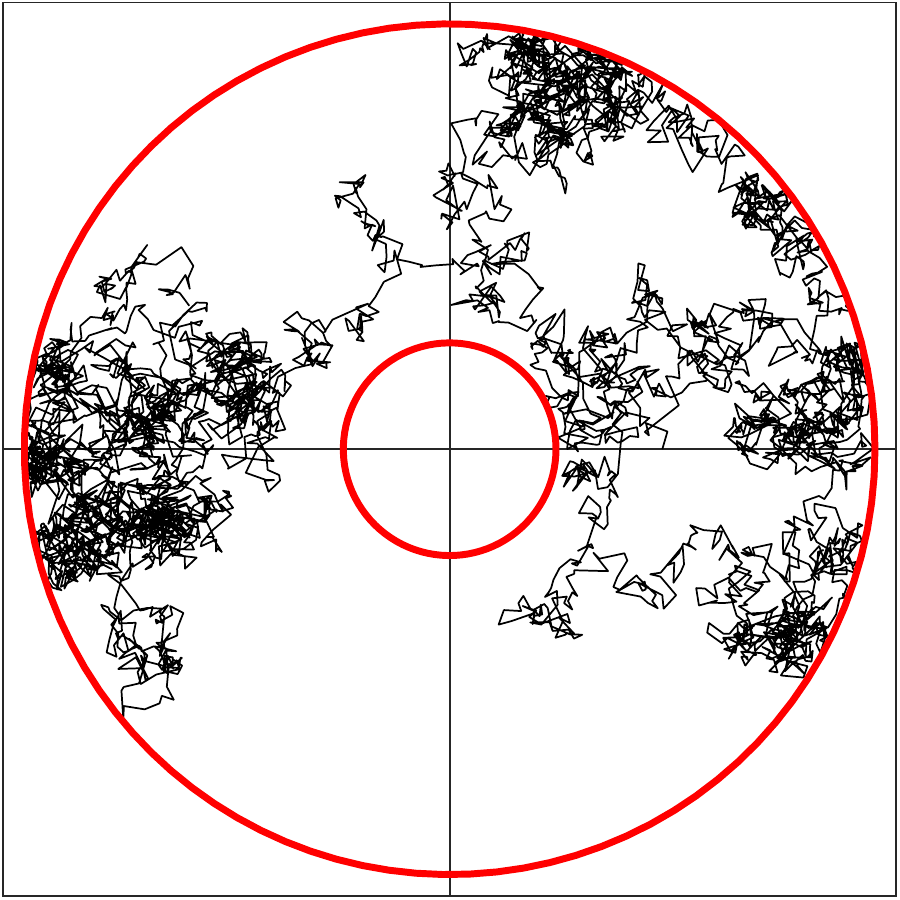}}
  \caption{(a) Winding of a Brownian particle around the origin, starting
    from~$\Rv(0)=(1,0)$, subjected to a central point vortex of
    strength~$\beta=1$.  (b) Winding around a disk of radius~$a=1/2$ with
    $\beta=3$; (c) Winding in an annulus with inner radius~$a=1/2$, outer
    radius~$b=2$, with~$\beta=1$.  Note the tendency of the path to wind
    counterclockwise because of the influence of the point vortex.}
  \label{fig:path}
\end{figure}

\section{Winding around a point}
\label{sec:pt}

For winding around the origin, the domain of \cref{eq:char} is
$0 < r < \infty$.  As a result, only Bessel functions of the first kind are
involved, so the solution to \cref{eq:fund_pde} is
\begin{equation}
  \P(r,\theta,t)
  =
  \frac{1}{2\pi}\int_{-\infty}^\infty \int_0^\infty
  \ee^{-\lambda^2t}\ee^{\imi\mu\theta}
  J_{\kd_\mu}(\lambda r)\,J_{\kd_\mu}(\lambda r_0)\,\lambda
  \dint \lambda \dint\mu\,,
  \label{eq:Ppoint}
\end{equation}
where~$\kd_\mu$ is defined in \cref{eq:char}, and $J_\nu$ is a Bessel
functions of the first kind of order~$\nu$.  Since we are interested in the
asymptotic behavior of $\P(r,\theta,t)$ as $t \rightarrow \infty$, the
$\lambda$ integral is dominated by small $\lambda$ (Watson's
Lemma~\cite{BenderOrszag}).  Therefore, we can choose $t$ large enough such
that
\begin{equation}
  J_{\kd_\mu}(\lambda r_0)
  \sim
  \frac{1}{\Gamma(1+\kd_\mu)}\left(\frac{\lambda r_0}{2}\right)^{\kd_\mu},
  \qquad
  \lambda r_0 \ll \sqrt{\lvert 1 + \kd_\mu\rvert}.
  \label{eq:Jsmallarg}
\end{equation}
The restriction~$\lambda r_0 \ll \sqrt{\lvert 1 + \kd_\mu\rvert}$ means that
we can use this to approximate~$J_{\kd_\mu}(\lambda r_0)$, but not
necessarily~$J_{\kd_\mu}(\lambda r)$, since~$r$ goes to infinity in the
integral.  With this approximation,~\cref{eq:Ppoint} becomes
\begin{align*}
  \P(r,\theta,t)
  &\sim \frac{1}{2\pi} \int_{-\infty}^\infty \int_0^\infty
  \ee^{-\lambda^2 t}
  \ee^{\imi\mu\theta}\, \frac{1}{\Gamma(1+\kd_\mu)}
  \left(\frac{\lambda r_0}{2}\right)^{\kd_\mu}\,
  J_{\kd_\mu}(\lambda r)\,\lambda \dint \lambda \dint \mu\\
  &= \frac{1}{2\pi}\int_{-\infty}^\infty \frac{1}{\Gamma(1+\kd_\mu)}
  \left(\frac{r_0}{2}\right)^{\kd_\mu} \ee^{\imi\mu\theta} \int_0^\infty
  \ee^{-\lambda^2 t}
  \lambda^{\kd_\mu+1}J_{\kd_\mu}(\lambda r) \dint\lambda\dint \mu\,.
\end{align*}
Then the winding angle distribution is given by
\begin{align}
  \W(\theta,t)
  &=
  \int_0^\infty \P(r,\theta,t)\,r \dint r\nonumber\\
  &\sim
  \frac{1}{2\pi}\int_{-\infty}^\infty
  \frac{1}{\Gamma(1+\kd_\mu)} \left(\frac{r_0}{2}\right)^{\kd_\mu}
  \ee^{\imi\mu\theta} \int_0^\infty \int_0^\infty
  \ee^{-\lambda^2 t}\lambda^{\kd_\mu+1}J_{\kd_\mu}(\lambda r)\,r
  \dint\lambda \dint r \dint\mu.
\end{align}
The~$r$ integral can be carried out analytically, and we have
\begin{align}
  \W(\theta,t)
  &\sim
  \frac{1}{2\pi}\int_{-\infty}^\infty
  \frac{\Gamma(1+\tfrac{\kd_\mu}{2})}{\Gamma(1+\kd_\mu)}
  \left(\frac{r_0}{2\sqrt{t}}\right)^{\kd_\mu}
  \ee^{\imi\mu\theta} \dint\mu\nonumber\\
  &\sim
  \frac{1}{2\pi}\int_{-\infty}^\infty
  \left(\frac{r_0\,\ee^{\gamma/2}}{2\sqrt{t}}\right)^{\kd_\mu}
  \ee^{\imi\mu\theta} \dint\mu\,,
  \label{eq:mu_int}
\end{align}
where~$\gamma = 0.577215\ldots$ is the Euler--Mascheroni constant.  In the
last step we expanded the integrand in small~$\kd_\mu$,
since~$(\nofrac{r_0}{\sqrt{t}}) \ll 1$.

Let~$A(t) = \sqrt{\nofrac{\beta}{8}}\,
\log\l(\nofrac{4t}{r_0^2\,\ee^\gamma}\r)$; then
\begin{equation}
  \W(\theta,t)
  \sim
  \frac{1}{2\pi}
  \int_{-\infty}^\infty
  \ee^{-A(t)\sqrt{\nofrac{2}{\beta}}\,\kd_\mu+\imi\mu\theta}\dint\mu\,.
\end{equation}
Since $A(t)\rightarrow \infty$ as $t\rightarrow \infty$, the integral is
dominated by small $\mu$. Note that from~\cref{eq:char},
as~$\mu \rightarrow 0$,
\begin{equation}
  \kd_\mu
  \sim
  \sqrt{\imi\beta\mu}
  =
  \sqrt{\frac{\beta|\mu|}{2}}\left(1 + \imi\sign\mu\right)
\end{equation}
where~$\sign$ is the signum function, leading to
\begin{align}
  \W(\theta,t)
  &=
  \frac{1}{2\pi} \int_{-\infty}^\infty
  \ee^{-A\sqrt{|\mu|}+\imi(\mu\theta - \sign(\mu) A\sqrt{|\mu|})}\dint\mu
  \nonumber\\
  &=
  \frac{1}{\pi}\int_{0}^\infty \ee^{-A\sqrt{\mu}}
  \cos(\mu\theta - A\sqrt{\mu})\dint\mu\nonumber\\
  &=
  \frac{1}{\sqrt{2\pi}\,A^2}\,
  \x^{-\nofrac{3}{2}}\,\ee^{-\nofrac{1}{2\x}}\, \chi_{(\x>0)},
  \qquad \x = \nofrac{\theta}{A^2}\,,
  \label{eq:small_mu_approx}
\end{align}
where $\chi$ is the indicator function.  We conclude that the normalized
winding angle at time $t \rightarrow \infty$ converges in distribution
according to
\begin{equation}
  \frac{8\Theta(t)}{\beta\log^2\!{\l(\nofrac{4t}{r_0^2\,\ee^{\gamma}}\r)}}
  \convdist
  \X,
  \qquad
  \p_\X(\x) = \frac{1}{\sqrt{2\pi}}\,
  \x^{-\nofrac{3}{2}}\,\ee^{-\nofrac{1}{2\x}}\,\chi_{(\x>0)}\,.
  \label{eq:ptlimit}
\end{equation}

\smallskip
\noindent\emph{Remarks}
\smallskip

\begin{enumerate}
\item All positive integer moments of $\W$ diverges, due to the probability of
  infinite winding around the point-like origin, as in Spitzer's
  law~\cref{eq:Spitzer}.
\item $\x^{-1}$ is Gamma-distributed:
  \begin{equation}
    \frac{\beta\log^2\!{\l(\nofrac{4t}{r_0^2\,\ee^\gamma}\r)}}{8\Theta(t)}
    \convdist
    \Gammadist(\tfrac{1}{2},\tfrac{1}{2}).
  \end{equation}
\item For~$\beta=0$ the leading order term of~$\kd_\mu$ is no
  longer~$\Order{\sqrt{\mu}}$, rendering \cref{eq:small_mu_approx} invalid. We
  need to go back to \cref{eq:mu_int} and use~$\kd_\mu \sim \lvert\mu\rvert$,
  which leads to
  \begin{align}
    \W(\theta,t)
    &\sim
    \frac{1}{\pi}
    \int_{0}^\infty \ee^{-\tfrac12\,\mu \log\l(\nofrac{4t}{r_0^2\,\ee^\gamma}\r)}
    \cos(\mu\theta)\dint\mu \nonumber \\
    &=
    \frac{1}{\pi}\,\frac{A}{A^2 + \theta^2}\,,
    \qquad
    A = \tfrac12\,\log\l(\nofrac{4t}{r_0^2\,\ee^\gamma}\r),
    \label{eq:GFimprovement}
  \end{align}
  which is Spitzer's law~\cref{eq:Spitzer} with the updated normalizer of
  Grosberg \& Frisch~\cite{Grosberg2003}.  Our calculation adds a
  correction~$\ee^\gamma$ to their normalizer.  (See \cref{apx:revisit}.)
\end{enumerate}

\begin{figure}
  \centering
  \includegraphics[width=.45\textwidth]{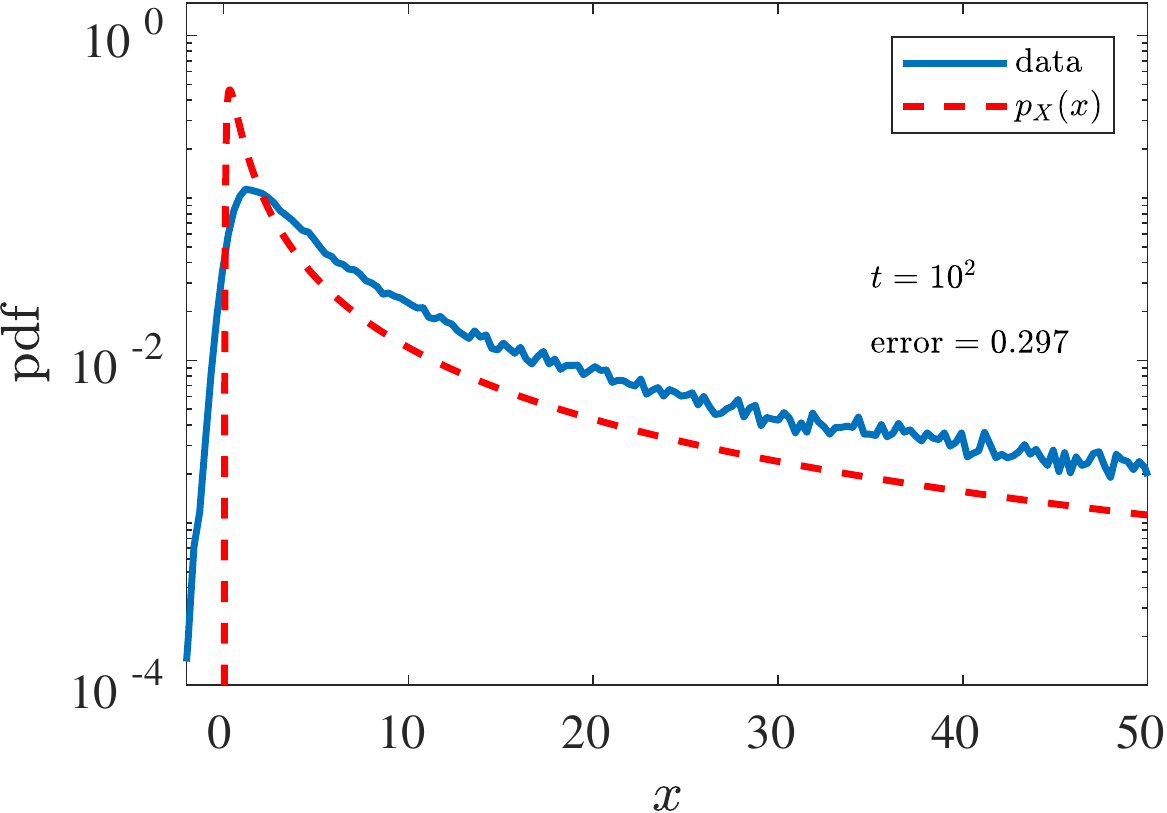}

  \includegraphics[width=.45\textwidth]{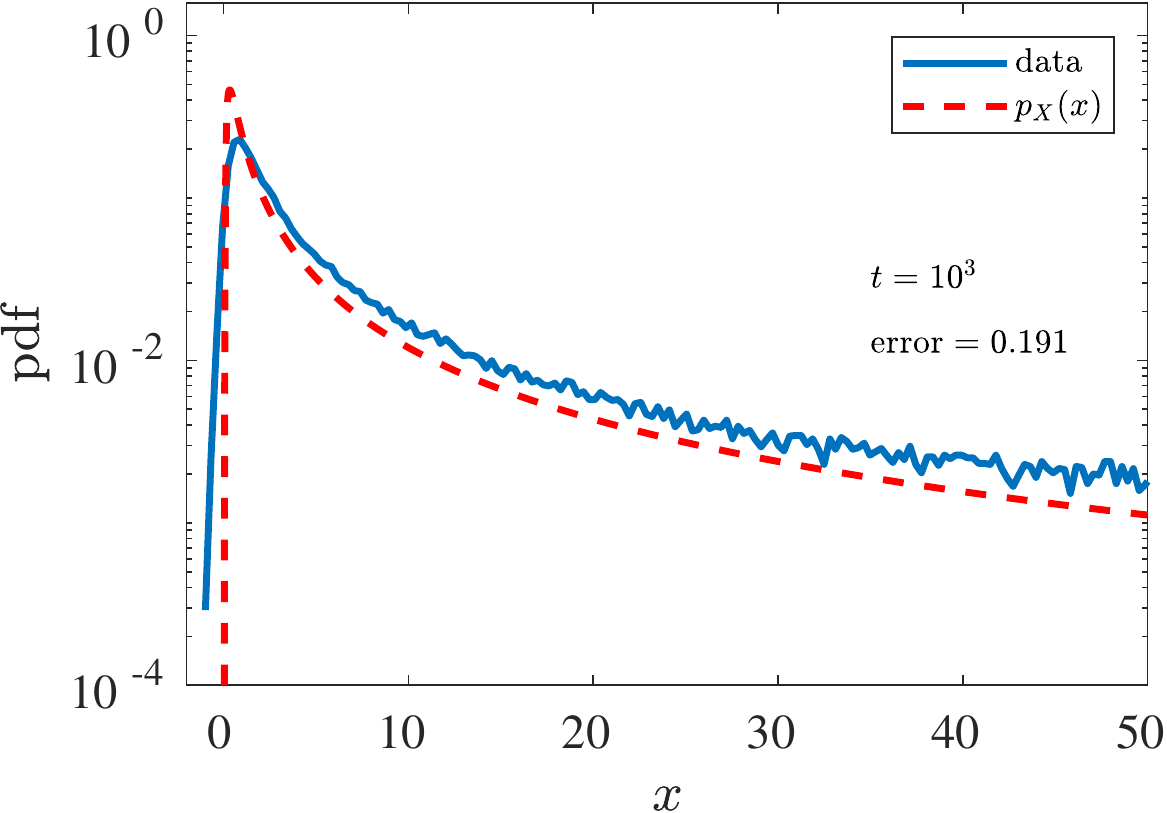}

  \includegraphics[width=.45\textwidth]{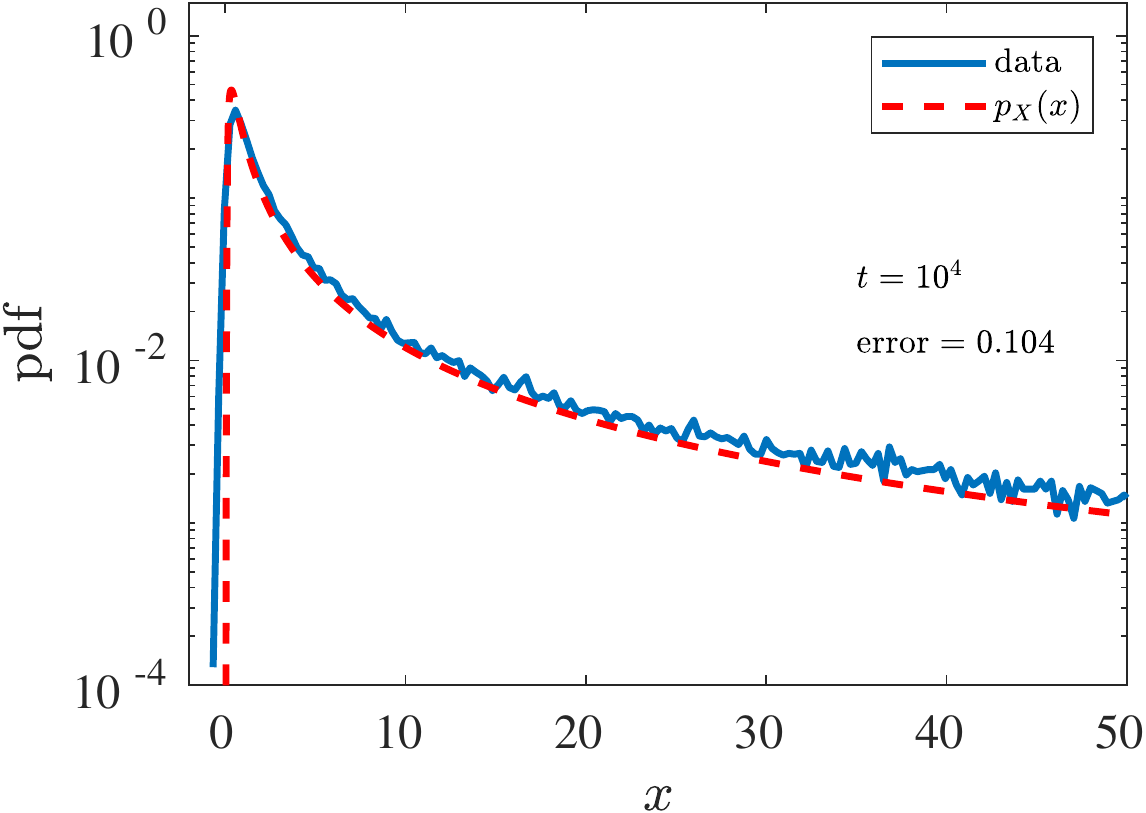}
  \caption{Numerical simulation of $10^5$ realizations of winding around the
    origin for a Brownian particle subjected to a point vortex.  Here
    $\beta=1$ and $r_0=1$.  The error is the $\Ltwo$ norm of the difference
    between the data and the limit distribution $\p_\X(\x)$ in
    \cref{eq:ptlimit}. We can see the convergence as $t \rightarrow \infty$.}
\end{figure}

\section{Winding around a disk}
\label{sec:inner_disk}

We now remove a disk of radius $a$ centered on the origin, with reflective
boundary conditions at its boundary.  The particle is in the
domain~\hbox{$a<|\Rv(t)|<\infty$} and so is prevented from coming close to the
origin.  The domain of~\cref{eq:char} is~\hbox{$a < r < \infty$}, with a
reflecting boundary condition~$\R'(a)=0$.  The bounded solution to
\cref{eq:char} can then be written
as~$\R(r) = Z_{\kd_\mu}(\lambda r, \lambda a)$, with
\begin{equation}
  Z_{\kd_\mu}(\lambda r, \lambda a)
  =
  \frac{J_{\kd_\mu}(\lambda r)\,Y'_{\kd_\mu}(\lambda a)
    - J'_{\kd_\mu}(\lambda a)\,Y_{\kd_\mu}(\lambda r)}
  {\sqrt{J'^2_{\kd_\mu}(\lambda a)+Y'^2_{\kd_\mu}(\lambda a)}}
  \label{eq:Z_func_def}
\end{equation}
where $\kd_\mu$ is as in~\cref{eq:char}, and $J_\nu$, $Y_\nu$ are Bessel
functions of the first and second kind of order $\nu$.  Thus the solution to
\cref{eq:fund_pde} is written as
\begin{equation}
  \P(r,\theta,t)
  =
  \frac{1}{2\pi}\int_{-\infty}^\infty \int_0^\infty
  \ee^{-\lambda^2 t}\,\ee^{\imi\mu\theta}
  \,Z_{\kd_\mu}(\lambda r,\lambda a)\,Z_{\kd_\mu}(\lambda r_0,\lambda a)
  \,\lambda \dint \lambda \dint\mu\,.
  \label{eq:inner_disk_solution}
\end{equation}
As in the previous \namecref{sec:pt}, $t\rightarrow \infty$ allows a
small-$\lambda$ approximation where we make use of the asymptotic forms
\begin{subequations}
\begin{alignat}{2}
  J_\nu(x) &\sim \frac{(x/2)^\nu}{\Gamma(1+\nu)},
  \qquad
  &Y_\nu(x) &\sim \frac{(x/2)^\nu}{\Gamma(1+\nu)} \cot \pi \nu -
  \frac{(x/2)^{-\nu}}{\Gamma(1-\nu)} \csc \pi \nu, \\
  J'_\nu(x) &\sim \frac{(x/2)^\nu \nu x^{-1}}{\Gamma(1+\nu)},
  \qquad
  &Y'_\nu(x) &\sim \frac{(x/2)^\nu \nu x^{-1}}{\Gamma(1+\nu)} \cot \pi \nu
  + \frac{(x/2)^{-\nu}\nu x^{-1}}{\Gamma(1-\nu)} \csc \pi \nu,
\end{alignat}
\label{eq:JYapprox}%
\end{subequations}
as $x\rightarrow 0$.  We thus have
\begin{equation}
  Z_{\kd_\mu}(\lambda r_0, \lambda a)
  \sim
  \frac{(r_0/a)^{\kd_\mu} + (r_0/a)^{-\kd_\mu}}
  {\sqrt{[(\lambda a/2)^{\kd_\mu}\Gamma(1-\kd_\mu)
      + (\lambda a/2)^{-\kd_\mu}\Gamma(1+\kd_\mu)]^2
      - 2\pi \kd_\mu \tan(\pi \kd_\mu/2)}}.
  \label{eq:Zsmalllambda}
\end{equation}
As we show in \cref{apx:smallk}, the $\mu$ integral is dominated by small
$\mu$, hence small $\kd_\mu$.  As $\kd_\mu\rightarrow 0$, we have
$\Gamma(1-\kd_\mu)\sim \ee^{\kd_\mu\gamma}$, where $\gamma$ is the
Euler--Mascheroni constant; $Z_{\kd_\mu}(\lambda r_0,\lambda a)$ further
simplifies
\begin{equation}
  Z_{\kd_\mu}(\lambda r_0,\lambda a)
  \sim
  \frac{\cosh(\kd_\mu\log(r_0/a))}{\cosh(\kd_\mu\log(\lambda \at/2))},
  \qquad
  \at \ldef a\ee^\gamma.
  \label{eq:Zsmalllambdasmallmu}
\end{equation}
For $Z_{\kd_\mu}(\lambda r,\lambda a)$ in \cref{eq:Z_func_def}, as
$\lambda\rightarrow 0$ we have
\begin{align}
  Z_{\kd_\mu}(\lambda r,\lambda a)
  &\sim
  \frac{J_{\kd_\mu}(\lambda r)\left((\lambda \at/2)^{\kd_\mu}\cot{\pi \kd_\mu}
      + (\lambda \at/2)^{-\kd_\mu}\csc{\pi \kd_\mu}\right)
    - Y_{\kd_\mu}(\lambda r)(\lambda \at/2)^{\kd_\mu}}
  {\sqrt{\left((\lambda \at/2)^{\kd_\mu}\cot{\pi \kd_\mu}
        + (\lambda \at/2)^{-\kd_\mu}\csc{\pi \kd_\mu}\right)^2
      + (\lambda \at/2)^{2\kd_\mu}}}\nonumber\\
  &\sim
  J_{\kd_\mu}(\lambda r) - Y_{\kd_\mu}(\lambda r)\,
  \frac{(\lambda \at/2)^{\kd_\mu}\sin{\pi \kd_\mu}}
  {(\lambda \at/2)^{\kd_\mu} + (\lambda \at/2)^{-\kd_\mu}}.
  \label{eq:Zkasym}
\end{align}

The time-asymptotic winding distribution is
\begin{equation}
  \W(\theta,t)
  \sim
  \frac{1}{2\pi} \int_{-\infty}^\infty \cosh(\kd_\mu\log(\nofrac{r_0}{a}))\,
  \ee^{\imi\mu\theta} \, \I_{\kd_\mu}(t) \dint \mu
\end{equation}
where
\begin{equation}
  \I_{\kd_\mu}(t)
  =
  \int_a^\infty \int_0^\infty \ee^{-\lambda^2 t}
  Z_{\kd_\mu}(\lambda r,\lambda a)
  \sech(\kd_\mu\log(\lambda \at/2))\,\lambda r
  \dint \lambda \dint r.
  \label{eq:lambda_and_r_integral0}
\end{equation}
In \cref{apx:Ikasym} we show that for small $\kd_\mu$
and~$t \rightarrow \infty$,
\begin{equation}
  \I_{\kd_\mu}(t) \sim \sech(\kd_\mu\log(\at/2\sqrt{t})).
  \label{eq:Ikasym}
\end{equation}
Hence,
\begin{align}
  \W(\theta,t)
  &=
  \frac{1}{2\pi} \int_{-\infty}^\infty
  \cosh(\kd_\mu\log(r_0/a))\, \I_{\kd_\mu}(t) \, \ee^{\imi\mu\theta} \dint\mu
  \nonumber\\
  &\sim \frac{1}{2\pi} \int_{-\infty}^\infty \cosh(\kd_\mu\log(r_0/a))
  \,\sech(\kd_\mu\log(\at/2\sqrt{t}))\,\ee^{\imi\mu\theta} \dint\mu
  \nonumber\\
  &= \frac{1}{2\pi A^2} \int_{-\infty}^\infty
  \cosh(\sqrt{\imi \tau}B)
  \sech(\sqrt{\imi \tau})\, \ee^{\imi\tau \x} \dint \tau
\end{align}
where
\begin{equation}
  A(t) = \tfrac12\sqrt{\beta}\log(4t/\at^2),
  \quad
  B(t) = \tfrac{\sqrt{\beta}}{A(t)}\log(r_0/a),
  \quad
  \tau = \mu A^2,
  \quad
  \x = \theta/A^2.
\end{equation}

The $\tau$ integral can be evaluated by computing residues at
\begin{equation}
  \tau_n = \imi\pi^2\bigl(n+\tfrac{1}{2}\bigr)^2,
  \qquad
  n= 1,2,3,\dots.
\end{equation}
We thus have
\begin{align}
  \W(\theta,t)
  &=
  \frac{2\pi \imi}{2\pi A^2}
  \sum_{n=0}^\infty \pi\imi\, (2n+1)\,(-1)^{n+1}
  \,\ee^{-(2n+1)^2\pi^2\x/4}
  \cos\bigl(\pi \bigl(n+\tfrac{1}{2}\bigr)B\bigr)\,\chi_{(\x>0)}\nonumber\\
  &\sim
  \frac{\pi}{A^2}
  \sum_{n=0}^\infty (-1)^n\,(2n+1)
  \,\ee^{-(2n+1)^2\pi^2\x/4}\,\chi_{(\x>0)},
\end{align}
where $\chi$ is the indicator function.  Introduce the second elliptic theta
function
\begin{equation}
  \vartheta_2(z,q) = 2\sum_{n=0}^\infty q^{(2n+1)^2/4}\cos((2n+1)z)
\end{equation}
and by convention,
\begin{equation}
  \vartheta'_2(z,q) \ldef \D{\vartheta_2}{z}(z,q).
\end{equation}
Then the asymptotic winding law for the random angle~$\Theta(t)$ can be
expressed as
\begin{equation}
  \frac{4\Theta(t)}{\beta\,\log^2({4t}/{a^2\,\ee^{2\gamma}})}
  \convdist
  \X,
  \qquad
  \p_\X(\x)
  =
  -\tfrac{\pi}{2}\,\vartheta_2'
  \bigl(\tfrac{\pi}{2}\,,\,\ee^{-\pi^2 \x}\bigr)\,\chi_{(\x>0)}\,.
  \label{eq:disklimit}
\end{equation}
Note that now all the moments exist.  The asymptotic
distribution~\cref{eq:disklimit} is compared to numerical simulations in
\cref{fig:inner_disk}.

\begin{figure}
  \centering
  \includegraphics[width=.45\textwidth]{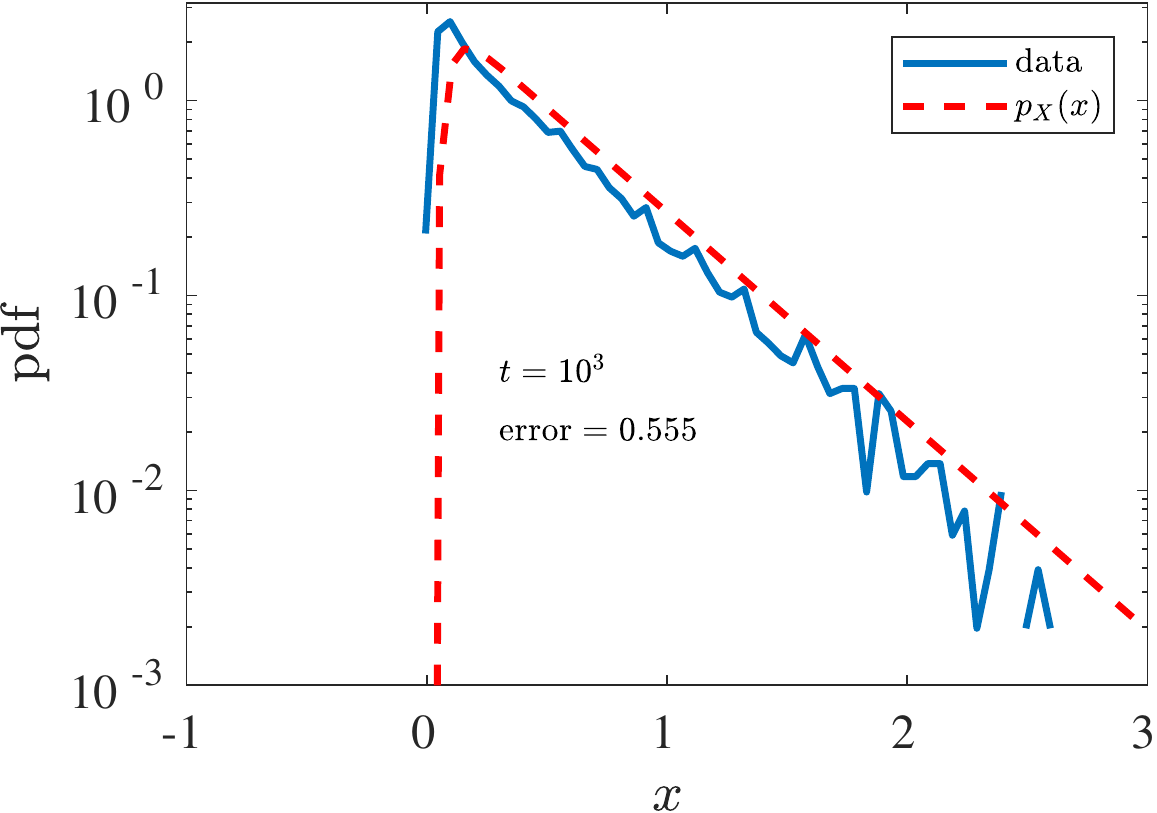}

  \includegraphics[width=.45\textwidth]{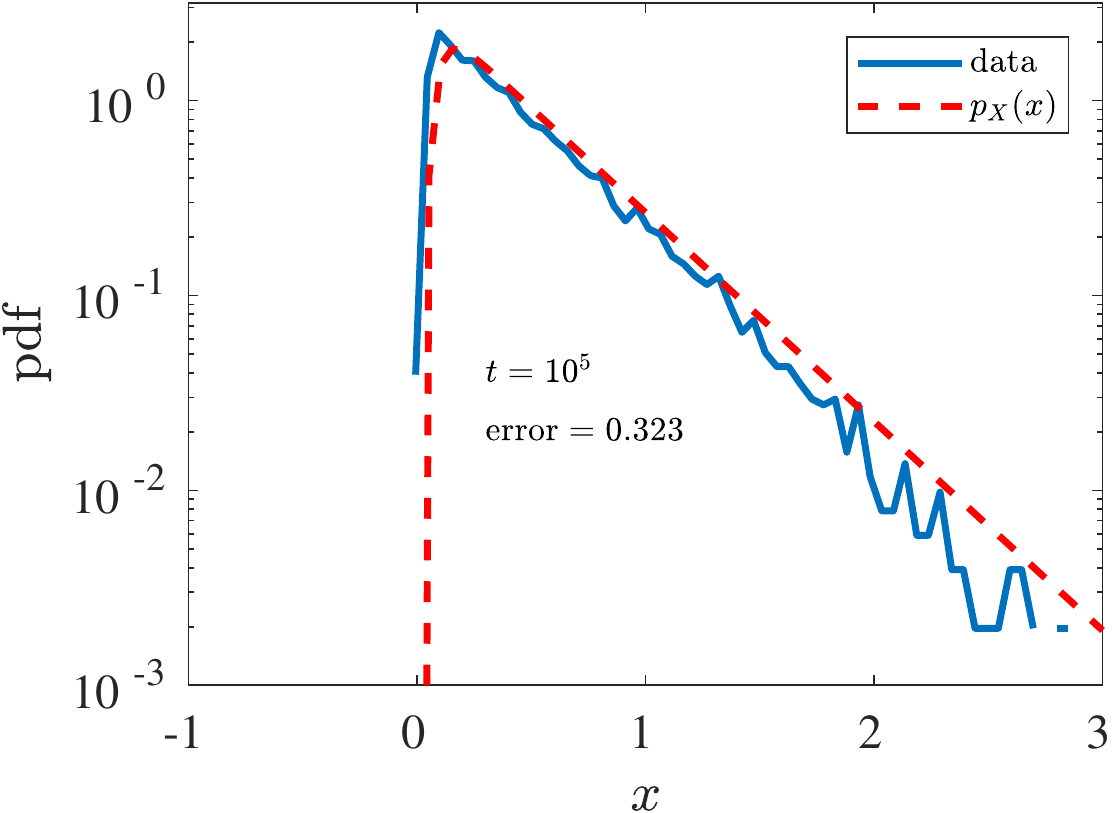}

  \includegraphics[width=.45\textwidth]{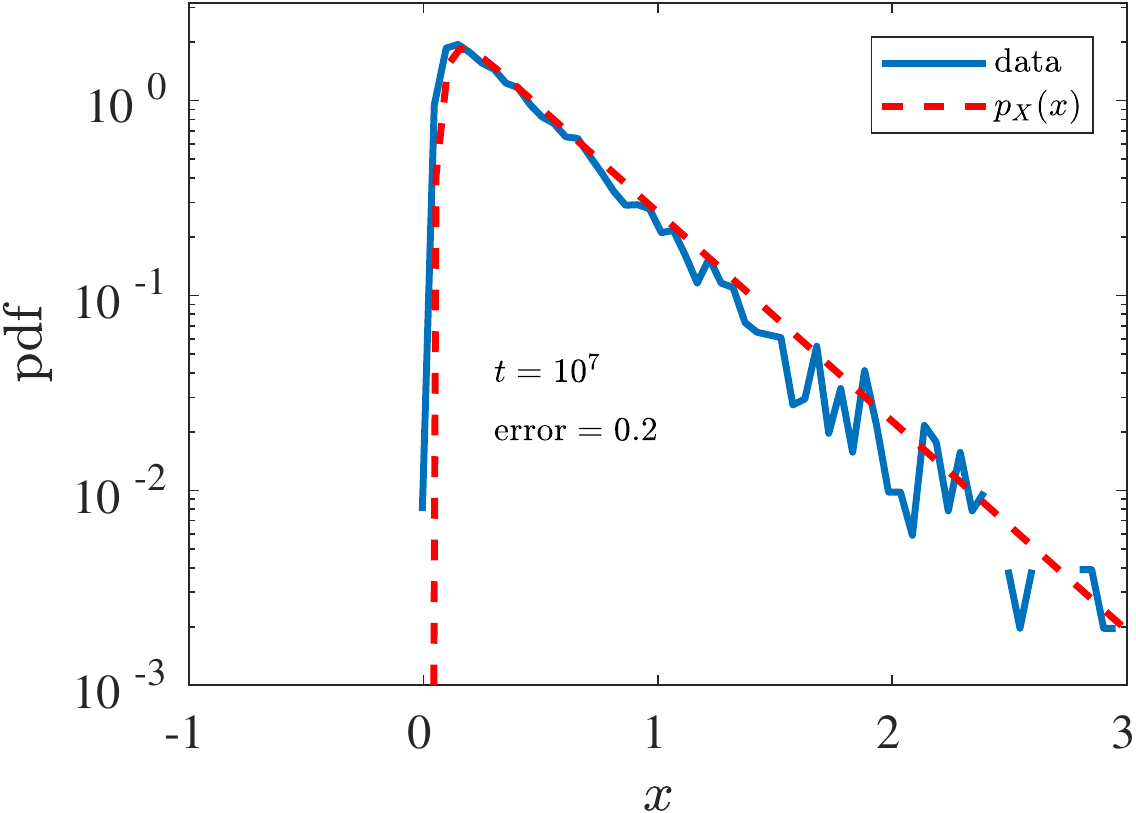}
  \caption{Numerical simulation of $10^4$ realizations of winding around a
    disk of radius~$a=1/10$ for a Brownian particle subjected to a point
    vortex of strength~$\beta=3$.  The error is the $\Ltwo$ norm of the
    difference between the data and the limit distribution $\p_\X(\x)$ in
    \cref{eq:disklimit}. We can see the convergence as
    $t \rightarrow \infty$.}
  \label{fig:inner_disk}
\end{figure}

\section{Winding in an annulus}
\label{sec:annulus}

Finally, consider the case where the particle is confined to an annulus
$a<|\Rv(t)|<b$, with both boundaries being reflective.  Then the solution to
\cref{eq:char} is modified accordingly to
\begin{equation}
  Z_{\kd_\mu}(\lambda r)
  =
  \frac{1}{m_{\lambda,\mu}}
  \left(-J_{\kd_\mu}(\lambda r)\, Y'_{\kd_\mu}(\lambda a)
    + J'_{\kd_\mu}(\lambda a)\, Y_{\kd_\mu}(\lambda r)\right)
\end{equation}
where~$\kd_\mu$ is as in~\cref{eq:char}, $m_{\lambda,\mu}$ is a normalization
factor with
\begin{subequations}
\begin{gather}
  \int_a^b Z^2_{\kd_\mu}(\lambda r)\, r\dint r = 1,
  \label{eq:normConstr} \\
\intertext{and}
  -J'_{\kd_\mu}(\lambda b)\, Y'_{\kd_\mu}(\lambda a)
  + J'_{\kd_\mu}(\lambda a)\, Y'_{\kd_\mu}(\lambda b)
  = 0,
  \label{eq:boundaryConstr}
\end{gather}
\end{subequations}
in order to satisfy the reflective boundary conditions
\begin{equation}
  \D{Z_{\kd_\mu}}{r}(\lambda a) = \D{Z_{\kd_\mu}}{r}(\lambda b) = 0.
\end{equation}
The solution to \cref{eq:fund_pde} now involves a discrete sum over~$\lambda$:
\begin{equation}
  \P(r,\theta,t)
  =
  \frac{1}{2\pi}\int_{-\infty}^\infty\sum_{\lambda}
  \ee^{-\lambda^2 t}\,\ee^{\imi \mu \theta}
  Z_{\kd_\mu}(\lambda r)\,Z_{\kd_\mu}(\lambda r_0)\dint\mu.
  \label{eq:densityFormula}
\end{equation}
The sum is over all solutions to \cref{eq:boundaryConstr} with positive real
part, since we require $\P(r,\theta,t)$ to stay finite as
$t\rightarrow \infty$ and $Z_\mu(\lambda r)$ is undefined for $\lambda =
0$. For fixed $\mu$, denote $\lambda_{\kd_\mu}$ the solution to
\cref{eq:boundaryConstr} with smallest positive real part. For large $t$, the
summation over $\lambda$ is dominated by the leading term,
\begin{equation}
  \ee^{-\lambda_{\kd_\mu}^2 t}\,\ee^{\imi \mu \theta}\,
  Z_{\kd_\mu}(\lambda_{\kd_\mu} r)\,Z_{\kd_\mu}(\lambda_{\kd_\mu} r_0).
\end{equation}
Using the approximations~\cref{eq:JYapprox} for $J_\nu$ and $Y_\nu$ for small
argument, we have
\begin{align*}
  Z_{\kd_\mu}(\lambda_{\kd_\mu} r)
  &\sim
  -\frac{\kd_\mu \csc \pi \kd_\mu}
  {m_{\lambda_{\kd_\mu},\kd_\mu}\lambda_{\kd_\mu}
    a \Gamma(1+\kd_\mu)\Gamma(1-\kd_\mu)}
  \left[\left(\frac{r}{a}\right)^{\kd_\mu}
    + \left(\frac{r}{a}\right)^{-\kd_\mu}\right] \\
  &=
  \frac{\cosh(\kd_\mu\log(r/a))}{C_{\kd_\mu}(a,b)}
\end{align*}
where the constant $C_{\kd_\mu}(a,b)$ is determined by
\cref{eq:normConstr}.  Hence,
\begin{align}
  \P(r,\theta,t)
  &\sim
  \frac{1}{2\pi}\int_{-\infty}^\infty  \ee^{-\lambda_{\kd_\mu}^2 t}\,\ee^{\imi \mu\theta}
  \,\frac{\cosh(\kd_\mu\log(r/a))\cosh(\kd_\mu\log(r_0/a))}
  {C_{\kd_\mu}^2(a,b)}\dint \mu\,.
  \label{eq:mu_int2}
\end{align}
Before proceeding further, we prove the following:
\begin{proposition}
  As $\kd_\mu\rightarrow 0$,
  \begin{enumerate}
  \item $\lambda_{\kd_\mu} \rightarrow 0$.
  \item $\displaystyle \lambda_{\kd_\mu}^2
    = \frac{2\kd_\mu^2}{b^2-a^2}\log(\nofrac{b}{a}) + o(\kd_\mu^2)$.
  \end{enumerate}
\end{proposition}
\begin{proof} For the first claim, let
  \begin{equation}
    \mathcal{L} = \frac{d}{dr}\left(r\frac{d}{dr}\right) - \mu^2,
    \qquad
    \mathcal{M} = -\imi\beta\mu\,.
  \end{equation}
  Then \cref{eq:char} is rewritten as
  \begin{equation}
    (\mathcal{L} + \mathcal{M})\,\R(r) + \lambda^2 r^2 \R(r) = 0.
    \label{eq:complexEV}
  \end{equation}
  Consider the eigenvalue problem
  \begin{equation}
    \mathcal{L}\,\phi_\nu(r) + \nu r^2 \phi_\nu(r) = 0,
    \quad a<r<b,
    \qquad
    \phi_\nu'(a) = \phi_\nu'(b) = 0.
    \label{eq:realEV}
  \end{equation}
  Sturm--Liouville theory states that the eigenvalues are real and countable,
  $\nu_0<\nu_1<\ldots<\nu_n<\ldots$, with
  $\lim_{n\rightarrow \infty} \nu_n = \infty$. By the regularity requirement
  of the eigenfunctions as $r\rightarrow \infty$, we have~$\nu_0>0$.  This
  smallest eigenvalue can be bounded by the Rayleigh quotient:
  \begin{equation}
    \nu_0
    \leq
    -\frac{\int_a^b\phi\,\mathcal{L}\,\phi \dint r}{\int_a^b\phi^2r^2\dint r}
    =
    \frac{\int_a^b (r\phi'^2(r) + \mu^2\phi(r))\dint r}
    {\int_a^b \phi^2(r) r^2 \dint r}
  \end{equation}
  where $\phi$ is a test function satisfying $\phi'(a) = \phi'(b) = 0$. By
  choosing $\phi(r) = 1$, we have
  \begin{equation}
    \nu_0 \leq \frac{3(b-a) }{b^3-a^3}\mu^2
    \rightarrow 0 ~\text{ as } \mu \rightarrow 0.
  \end{equation}
  Since $\mathcal{L}$ and $\mathcal{M}$ commute, by
  \cite[p.~209]{Kato2013perturbation} the distance between the spectrum of
  $\mathcal{L}+\mathcal{M}$ and that of $\mathcal{L}$ is bounded by
  $\lVert\mathcal{M}\rVert = \beta\mu\rightarrow 0$.  Hence,
  \begin{equation}
    |\lambda_{\kd_\mu}^2 - \nu_0|
    \leq
    \mu \rightarrow 0 \quad\text{ as }\quad \mu \rightarrow 0
  \end{equation}
  and thus~$\lambda_{\kd_\mu}^2 \rightarrow \nu_0 \rightarrow 0$ as
  $\mu \rightarrow 0$, which proves the first part.

  For the second part, define
  \begin{equation}
    g_{\kd_\mu}(x) \ldef \nofrac{Y'_{\kd_\mu}(x)}{J'_{\kd_\mu}(x)}\,.
  \end{equation}
  Then \cref{eq:boundaryConstr} is rewritten as
  \begin{equation}
    g_{\kd_\mu}(\lambda a) = g_{\kd_\mu}(\lambda b)
  \end{equation}
  and using small-$x$ asymptotics
  \begin{align*}
    g_{\kd_\mu}(x)
    &\sim
    \cot(\pi \kd_\mu)
    + \left(\frac{x}{2}\right)^{-2\kd_\mu}
    \frac{\Gamma(\kd_\mu)\Gamma(1+\kd_\mu)}{\pi}
    + \left(\frac{x}{2}\right)^{-2\kd_\mu+2}
    \frac{2(2-\kd_\mu^2)\, \Gamma^2(\kd_\mu)}{\pi (1-\kd_\mu^2)} \\
    &=
    \cot(\pi \kd_\mu)
    + \left(\frac{x}{2}\right)^{-2\kd_\mu}
    \frac{\Gamma^2(\kd_\mu)}{\pi}
    \left[\kd_\mu
      + \left(\frac{x}{2}\right)^2\frac{2(2-\kd_\mu^2)}{(1-\kd_\mu^2)}
    \right].
  \end{align*}
  So $\lambda_{\kd_\mu}$ satisfies
  \begin{equation}
    a^{-2\kd_\mu}
    \left[\kd_\mu + \left(\frac{\lambda_{\kd_\mu} a}{2}\right)^2
      \frac{2(2-\kd_\mu^2)}{(1-\kd_\mu^2)}\right]
    =
    b^{-2\kd_\mu}
    \left[\kd_\mu + \left(\frac{\lambda_{\kd_\mu} b}{2}\right)^2
      \frac{2(2-\kd_\mu^2)}{(1-\kd_\mu^2)}\right].
  \end{equation}
  Taking $\kd_\mu \rightarrow 0$, we have
  \begin{equation}
    \left(\nofrac{b}{a}\right)^{2\kd_\mu}
    \sim
    \frac{\kd_\mu + \lambda_{\kd_\mu}^2 b^2}{\kd_\mu + \lambda_{\kd_\mu}^2 a^2}
  \end{equation}
  and then expanding in $\kd_\mu$ and comparing coefficients yields
  \begin{equation}
    \lambda_{\kd_\mu}^2
    =
    \frac{2\kd_\mu^2}{b^2-a^2}\log\left(\nofrac{b}{a}\right) + o(\kd_\mu^2).
  \end{equation}
\end{proof}

Now we return to \cref{eq:mu_int2}.  Since $t$ is large, the integral is
dominated by small $\lambda_{\kd_\mu}$, hence small values of $\kd_\mu$ and
$\mu$. It can be shown that for small $\mu$, and hence small $\kd_\mu$, the
quantity
\begin{equation}
  \frac{\cosh(\kd_\mu
    \log(r/a))\cosh(\kd_\mu\log(r_0/a))}{C^2(\kd_\mu,a,b)}
\end{equation}
does not depend on $\kd_\mu$, due to the constraint \cref{eq:normConstr}.  In
fact, for small $\kd_\mu$,
\begin{equation}
  \frac{\cosh(\kd_\mu\log(r/a))}{C_{\kd_\mu}(a,b)}
  \longrightarrow \sqrt\frac{2}{b^2-a^2}
  \quad
  \text{as }
  \quad
  \kd_\mu\rightarrow 0.
\end{equation}
Hence
\begin{equation}
  \P(r,\theta,t)
  \sim
  \frac{1}{\pi(b^2-a^2)}\int_{-\infty}^\infty
  \ee^{-\frac{2 t}{b^2-a^2}\log(\nofrac{b}{a})(\mu^2+\imi \beta \mu)}\ee^{\imi \mu\theta}
  \dint\mu.
\end{equation}
This gives the winding angle distribution
\begin{equation}
  \W(\theta,t)
  \sim
  \frac{1}{2\pi}\int_{-\infty}^\infty
  \ee^{-A\mu^2}\ee^{\imi\mu(\theta-A\beta)}\dint \mu
  = \frac{1}{\sqrt{4A}}\exp\{-\nofrac{(\theta-A\beta)^2}{4A}\}
\end{equation}
where
\begin{equation}
  A(t) = \frac{2 t}{b^2-a^2}\log(\nofrac{b}{a}).
\end{equation}
Thus, the asymptotic winding law for~$\Theta(t)$ as~$t \rightarrow \infty$ is
given by
\begin{equation}
  \frac{\Theta(t)-A(t)\beta}{\sqrt{2A(t)}}
  \convdist
  N(0,1)
  \label{eq:annnuluslimit}
\end{equation}
where~$N(0,1)$ denotes the standard Gaussian distribution.  We compare this to
numerical simulations in \cref{fig:annulus_drift}.  In this case there is a
drift of the angle at a linear rate~$\beta A(t)$, since the confinement of the
particle causes it to be strongly affected by the point vortex.  The no-drift
case~$\beta=0$ is consistent with the Gaussian prediction of Comtet et
al.~\cite{Comtet1993b} and the recent result of Geng \&
Iyer~\cite{Geng2018_preprint}.

\begin{figure}
  \centering
  \includegraphics[scale = .75]{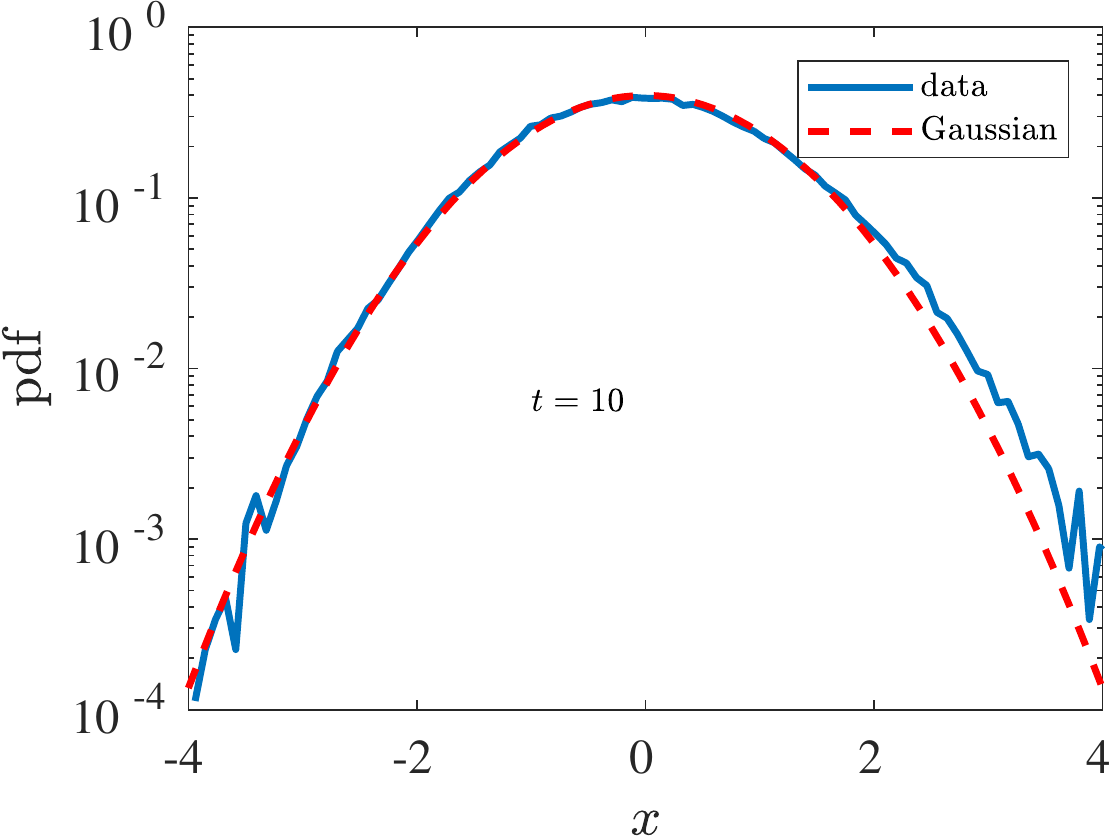}
  \caption{Numerical simulation of $10^5$ realizations of winding in an
    annulus with inner radius~$a=1/2$ and outer radius~$b=2$, with vortex
    strength~$\beta=1$.  Compared to the unbounded cases, the convergence
    in~$t$ is very rapid.  Note that there is a slight asymmetry biased in the
    direction of the drift, which is not captured at this order of the
    expansion.}
  \label{fig:annulus_drift}
\end{figure}

\appendix

\section{Small-$\kd_\mu$ asymptotics of integral}
\label{apx:smallk}

\mathnotation{\cc}{c}

We justify that the $\mu$ integral in~\cref{eq:inner_disk_solution} is
dominated by small $\kd_\mu$ when $t\rightarrow \infty$, or equivalently when
$\lambda \rightarrow 0$.  By symmetry, we need only consider the case where
$\mu>0$.  For some fixed small $\delta > 0$, we will show that as
$\lambda\rightarrow 0$,
\begin{subequations}
\begin{align}
  \I_{<\delta} &=
  \left\lvert\int_0^{\delta}
  \ee^{\imi\mu\theta}
  \,Z_{\kd_\mu}(\lambda r,\lambda a)\,Z_{\kd_\mu}(\lambda r_0,\lambda a)
  \dint \mu \right\rvert
  = \Order{\nofrac{1}{\log\lambda^{-1}}},
  \label{eq:Ideltasmall}
  \\
  \I_{>\delta} &=
  \left\lvert\int_{\delta}^\infty
  \ee^{\imi\mu\theta}
  \,Z_{\kd_\mu}(\lambda r,\lambda a)\,Z_{\kd_\mu}(\lambda r_0,\lambda a)
  \dint \mu \right\rvert
  = \Order{\nofrac{\lambda^{\sqrt{2}\delta}}{\log\lambda^{-1}}}.
  \label{eq:Ideltabig}
\end{align}
\end{subequations}
Hence, the ratio~$\I_{<\delta}/\I_{>\delta} \rightarrow \infty$
as~$\lambda\rightarrow0$ and the integral is dominated by small
$\mu < \delta$.  In other words, we show that for any fixed small $\delta>0$,
there is a $\lambda$ small enough to ensure that the $\mu$-integral is
dominated by $\mu < \delta$.

For small $\lambda$ we can use the expansion~\cref{eq:Zsmalllambda}
for~$Z_{\kd_\mu}(\lambda r_0, \lambda a)$.  Since
$\kd_\mu = \sqrt{\mu^2+\imi\beta\mu}$ with~$\beta > 0$ and~$r_0 > a$, we have,
as $\mu \rightarrow \infty$,
\begin{equation}
  \kd_\mu\sim \mu + \imi\beta/2,
  \qquad
  \Gamma(1+\kd_\mu) \sim \sqrt{2\pi\kd_\mu}\,(\kd_\mu)^{\kd_\mu}\ee^{-\kd_\mu},
  \qquad
  \bigl\lvert(r_0/a)^{\kd_\mu}\bigr\rvert \rightarrow \infty,
\end{equation}
the second of which is Stirling's approximation.  Hence, one of the terms in
the denominator of~\cref{eq:Zsmalllambda} is
\begin{equation}
  (\lambda a/2)^{-\kd_\mu}\Gamma(1+\kd_\mu)
  \sim
  \sqrt{2\pi\kd_\mu}\,(2\kd_\mu/\lambda a\ee)^{\kd_\mu}
  \label{eq:lambda-a2-kd_m}
\end{equation}
whose magnitude goes to~$\infty$ rapidly with increasing~$\mu$.  Euler's
reflection formula $\Gamma(z)\Gamma(1-z) = \pi/\sin(\pi z)$ allows us to show
\begin{equation}
  \lvert(\lambda a/2)^{\kd_\mu}\Gamma(1-\kd_\mu)\rvert
  \lesssim
  \frac{\pi}{\sinh(\pi\beta/2)}\,
  \l\lvert
  \frac{(\lambda a/2)^{\kd_\mu}}{\Gamma(\kd_\mu)}\r\rvert
  \rightarrow 0.
\end{equation}
The final term in the denominator of~\cref{eq:Zsmalllambda} is
\begin{equation}
  \lvert 2\pi \kd_\mu \tan(\pi \kd_\mu/2)\rvert
  \lesssim
  2\pi \lvert \kd_\mu\rvert \coth(\pi\beta/4)
\end{equation}
which does not go to zero as~$\mu \rightarrow \infty$, but can be neglected
compared to the square of~\cref{eq:lambda-a2-kd_m}.  Moreover, it can also be
neglected compared to the next-order asymptotic improvement of Stirling's
formula; since it is the only negative term, we will obtain an asymptotic
lower bound.  Referring back to~\cref{eq:Zsmalllambda}, there exists $N$ large
enough such that whenever $\mu>N>\delta$, we have
$ \bigl\lvert(r_0/a)^{\kd_\mu} + (r_0/a)^{-\kd_\mu}\bigr\rvert \lesssim
2\bigl\lvert(r_0/a)^{\kd_\mu} \bigr\rvert $ for~$r_0 > a$, and
\begin{multline*}
  \sqrt{[(\lambda a/2)^{\kd_\mu}\Gamma(1-\kd_\mu)
      + (\lambda a/2)^{-\kd_\mu}\Gamma(1+\kd_\mu)]^2
      - 2\pi \kd_\mu \tan(\pi \kd_\mu/2)} \\
    \geq
    \bigl\lvert
    \sqrt{2\pi}\,(\kd_\mu)^{\nofrac{1}{2}}(2\kd_\mu/\lambda a\ee)^{\kd_\mu}
    \bigr\rvert.
\end{multline*}
Putting these together, we have the asymptotic bound
\begin{equation}
  \lvert Z_{\kd_\mu}(\lambda r_0,\lambda a)\rvert
  \leq
  \frac{2}{\sqrt{2\pi\lvert\kd_\mu\rvert}}
  \bigl\lvert (\lambda r_0\ee/2\kd_\mu)^{\kd_\mu} \bigr\rvert
\end{equation}
for~$\mu > N$.

We thus have the estimate
\begin{align*}
  \int_{N}^\infty
  \left\lvert Z_{\kd_\mu}(\lambda r,\lambda a)\,
    Z_{\kd_\mu}(\lambda r_0,\lambda a)\right\rvert \dint\mu
  &\leq
  \frac{2}{\pi}
  \int_{N}^\infty
  \left\lvert\frac{\lambda^2rr_0\ee^2}{4\kd_\mu^2}
  \right\rvert^{\Re(\kd_\mu)}\!\frac{1}{\lvert\kd_\mu\rvert} \dint\mu\\
  &\leq
  C_1(\beta) \int_{N}^\infty
  \left\lvert\frac{\lambda^2rr_0\ee^2}{4\mu^2}
  \right\rvert^\mu\,\frac{\!\dint\mu}{\mu} \\
  &\leq
  C_1(\beta) \int_{\delta}^\infty
  \left\lvert\frac{\lambda^2rr_0\ee^2}{4\delta^2}
  \right\rvert^\mu\,\frac{\!\dint\mu}{\mu}\\
  &\leq
  C_2(\delta,\beta,r,r_0) \frac{\lambda^{2\delta}}{\log\lambda^{-1}},
\end{align*}
where in the last step we used the asymptotic form for the incomplete Gamma
function.  The real part of $\kd_\mu$ is bounded from below for
all~$\mu \ge 0$:
\begin{equation}
  \Re(\kd_\mu)
  =
  (\mu^2(\mu^2 + \beta^2))^{1/2}\,\cos\l(\tfrac12\arg(\mu + \imi\beta)\r)
  \geq
  \mu/\sqrt{2}.
\end{equation}
On the interval~$\mu \in [\delta,N]$, since~$\mu$ is bounded the term
$(\lambda a /2)^{-\kd_\mu}$ dominates for small~$\lambda$ in the denominator
of \cref{eq:Zsmalllambda}; therefore we can bound the
integral~\cref{eq:Ideltabig} as
\begin{align*}
  \I_{>\delta}
  &\leq
  \int_{\delta}^\infty
  \left\lvert Z_{\kd_\mu}(\lambda r,\lambda a)\,
    Z_{\kd_\mu}(\lambda r_0,\lambda a)\right\rvert \dint\mu\\
  &\leq
  C_3(\delta,\beta,r,r_0)\int_{\delta}^N
  \lambda^{2\Re(\kd_\mu)} \dint\mu
  +
  \int_{N}^\infty
  \left\lvert Z_{\kd_\mu}(\lambda r,\lambda a)\,
    Z_{\kd_\mu}(\lambda r_0,\lambda a)\right\rvert \dint\mu\\
  &\leq
  C_3(\delta,\beta,r,r_0)\int_{\delta}^N
  \lambda^{\sqrt{2}\mu}\dint\mu
  +
  \int_{N}^\infty
  \left\lvert\frac{\lambda^2rr_0\ee^2}{4\kd_\mu^2}
  \right\rvert^{\Re(\kd_\mu)}\!\frac{1}{\lvert\kd_\mu\rvert} \dint\mu\\
  &\leq
  C_4(\delta,\beta,r,r_0)\frac{\lambda^{\sqrt{2}\delta}}{\log\lambda^{-1}}
  +
  C_2(\delta,\beta,r,r_0) \frac{\lambda^{2\delta}}{\log\lambda^{-1}}\\
  &=
  \Order{\nofrac{\lambda^{\sqrt{2}\delta}}{\log\lambda^{-1}}}.
\end{align*}

Finally we show~\eqref{eq:Ideltasmall}.  By choosing~$\delta$ small enough and
noting that $a$, $r$, $r_0$, $\theta$ are all fixed, a small-$\mu$
approximation can be applied to go from~\cref{eq:Zsmalllambda}
to~\cref{eq:Zsmalllambdasmallmu}, and then
\begin{equation}
  Z_{\kd_\mu}(\lambda r_0,\lambda a)
  \sim
  \frac{\cosh(\kd_\mu\log(r_0/a))}{\cosh(\kd_\mu\log(\lambda \at/2))}
  \sim
  \sech(\sqrt{\beta\mu/2}\,(1+\imi)\log(\lambda \at/2)).
  \label{eq:Zdelta}
\end{equation}
Therefore, denoting $\cc=\sqrt{\beta/2}\log(2/\lambda\at)$, we have
\begin{align}
  \I_{<\delta}
  &\sim
  \left\lvert\int_0^{\delta}
  \sech^2(\sqrt{\beta\mu/2}\,(1+\imi)\log(\lambda \at/2))
  \dint \mu \right\rvert \nonumber\\
  &=
  \frac{ (1-\imi)\cc\sqrt{\delta} \tanh((1-\imi)\cc\sqrt{\delta})
  +
  \imi \log(\cosh((1-\imi)\cc\sqrt{\delta}))}{\cc^2}\nonumber\\
  &\sim
  \frac{\cc\sqrt{\delta} - \pi/2
    + \imi \log(\sqrt{2}/2\sin(\cc\sqrt{\delta}+\pi/4)}{\cc^2}\nonumber\\
  &=
  \Order{\nofrac{1}{\log\lambda^{-1}}}.
\end{align}

\pagebreak[1]

\section{Small-$\kd_\mu$ asymptotics of $\I_{\kd_\mu}(t)$}
\label{apx:Ikasym}

In this \namecref{apx:Ikasym} we derive an asymptotic form
for~$\I_{\kd_\mu}(t)$ defined by~\cref{eq:lambda_and_r_integral0}, valid for
small $\kd_\mu$ and~$t \rightarrow \infty$.  We drop the~$\mu$ subscript
on~$\kd_\mu$ to lighten the notation, since we only need to know that~$\kd$ is
small.

Expanding~$Z_\kd(\lambda r,\lambda a)$ in~\cref{eq:lambda_and_r_integral0}
according to~\cref{eq:Zkasym}, we have
\begin{multline}
  \I_\kd(t)
  \sim
  \int_a^\infty \int_0^\infty \ee^{-\lambda^2 t}
  \left(J_\kd(\lambda r) - Y_\kd(\lambda r)\,
    \frac{(\lambda \at/2)^{\kd}\sin{\pi \kd}}
    {(\lambda \at/2)^{\kd}+(\lambda \at/2)^{-\kd}} \right) \\
  \times \sech(\kd\log(\lambda \at/2))\,
  \lambda r\dint \lambda \dint r.
  \label{eq:lambda_and_r_integral}
\end{multline}
Now make the change of variable
\begin{equation}
  \lambda = y/L, \qquad r = Lz/y,
  \quad\text{where}\quad L(t) = \sqrt{t},
\end{equation}
in~$\I_\kd(t)$.  We can expand $\sech(\kd\log(\at y/2L))$ when $(\at y/2L)<1$,
since $\Re(\kd)>0$:
\begin{equation}
  \sech(\kd\log\tfrac{\at y}{2L})
  =
  2\sum_{n=0}^\infty (-1)^n \left(\frac{\at y}{2L}\right)^{(2n+1)\kd}\,.
\end{equation}
The condition $(\at y/2L) < 1$ is satisfied since
\cref{eq:lambda_and_r_integral} is effectively truncated at finite~$y$, due to
the factor $\ee^{-y^2}$, so essentially by dominated convergence
\begin{multline*}
  \I_\kd(t)
  =
  2\sum_{n=0}^\infty (-1)^n \left(\frac{\at}{2L}\right)^{(2n+1)\kd}
  \int_0^\infty \int_{ay/L}^\infty \ee^{-y^2} y^{(2n+1)\kd-1}
  J_\kd(z)\,z \dint z \dint y \\
  - 2\sin{\pi \kd}
  \sum_{n=0}^\infty (-1)^n(n+1)
  \left(\frac{\at}{2L}\right)^{(2n+3)\kd}
  \int_0^\infty \int_{ay/L}^\infty \ee^{-y^2} y^{(2n+3)\kd-1}
  Y_\kd(z)\,z \dint z \dint y.
\end{multline*}
It can be shown that as~$L \rightarrow \infty$
\begin{equation}
  \int_0^\infty \int_{ay/L}^\infty
  \ee^{-y^2}y^{(2n+1)\kd-1}J_\kd(z)\,z \dint z\dint y
  \longrightarrow
  \tfrac{\kd}{2}\,\Gamma\l(\tfrac{2n+1}{2}\kd\r)
\end{equation}
and
\begin{equation}
  \int_0^\infty \int_{ay/L}^\infty \ee^{-y^2}y^{(2n+1)\kd-1}
  Y_\kd(z)\,z \dint z\dint y
  \longrightarrow
  \tfrac{\kd}{2}\,\Gamma\l(\tfrac{2n+1}{2}\kd\r)\cot \tfrac{\pi \kd}{2}
\end{equation}
for $0<\Re(\kd)<2$, which is satisfied since $\kd$ is small.  Hence,
\begin{multline}
  \I_\kd(t)
  =
  2\sum_{n=0}^\infty (-1)^n
  \left(\frac{\at}{2L}\right)^{(2n+1)\kd}
  \tfrac{\kd}{2}\,\Gamma\l(\tfrac{2n+1}{2}\kd\r) \\
  - 2\sin(\pi \kd)\cot(\pi \kd/2) \sum_{n=0}^\infty
  (-1)^n(n+1) \left(\frac{\at}{2L}\right)^{(2n+3)\kd}
  \tfrac{\kd}{2}\,\Gamma\l(\tfrac{2n+3}{2}\kd\r).
\end{multline}
Since for small $\kd$,
\begin{equation}
  \frac{2n+1}{2}\,\kd\,\Gamma\l(\frac{2n+1}{2}\kd\r) \sim 1,
  \qquad
  \sin(\pi \kd)\cot(\pi \kd/2)\sim 2,
\end{equation}
we have,
\begin{equation}
  \I_\kd(t)
  \sim
  2\sum_{n=0}^\infty (-1)^n
  \tfrac{1}{2n+1}
  \left(\frac{\at}{2L}\right)^{(2n+1)\kd}
  - 4\sum_{n=0}^\infty
  (-1)^n \tfrac{n+1}{2n+3} \left(\frac{\at}{2L}\right)^{(2n+3)\kd}.
\end{equation}
Reindexing the second sum then allows us to combine the sums to obtain
\begin{align}
  \I_\kd(t)
  &\sim
  2\sum_{n=0}^\infty (-1)^n
  \left(\frac{\at}{2L}\right)^{(2n+1)\kd}\nonumber\\
  &= \sech(\kd\log(\at/2L)).
\end{align}
which is \cref{eq:Ikasym}.

\section{Free Brownian motion revisited}
\label{apx:revisit}

In this \namecref{apx:revisit} we use our results to slightly improve the
normalizer used by Grosberg \& Frisch~\cite{Grosberg2003} for the free
Brownian motion cases~$\beta=0$ (i.e., without tangential drift), which makes
the normalizer asymptotically correct to~$\Order{1}$ for large~$t$.  Numerical
simulations confirm a modest improvement.

For the case of winding around a point, we already derived
in~\cref{eq:GFimprovement} the correction to Spitzer and Grosberg \& Frisch's
result,
\begin{equation}
  \frac{2\Theta(t)}{\log(4t/r_0^2\ee^\gamma)}
  \convdist
  \X,
  \qquad
  \p_\X(\x) = \frac{1}{\pi} \frac{1}{1+\x^2}\,.
  \label{eq:GFimprovement2}
\end{equation}
The only difference is the~$\ee^\gamma$ factor in the normalizer.  Similarly,
the calculation in \cref{sec:inner_disk} for~$\beta=0$ can be shown to give
\begin{equation}
  \frac{2\Theta(t)}{\log(4t/a^2\,\ee^\gamma)}
  \convdist
  \X,
  \qquad
  \p_\X(\x) = \tfrac{1}{2} \sech (\pi \x/2).
  \label{eq:GFsechimprovement}
\end{equation}
The reason for the missing factor in the normalizer of Grosberg \& Frisch can
be traced to their use of the asymptotic form~\cref{eq:Jsmallarg} for
the~$J_{\kd_\mu}(\lambda r)$ in~\cref{eq:Ppoint}.  This renders the~$r$
integral divergent at infinity.  To remedy this, the authors truncate the
integral at~$r \sim \sqrt{4t}$, which is reasonable but in doing so they lose
a constant factor of~$\ee^\gamma$ in the normalizer.  Our approach avoids this
and is asymptotically valid.

\begin{figure}
  \centering
  \includegraphics[scale = .45]{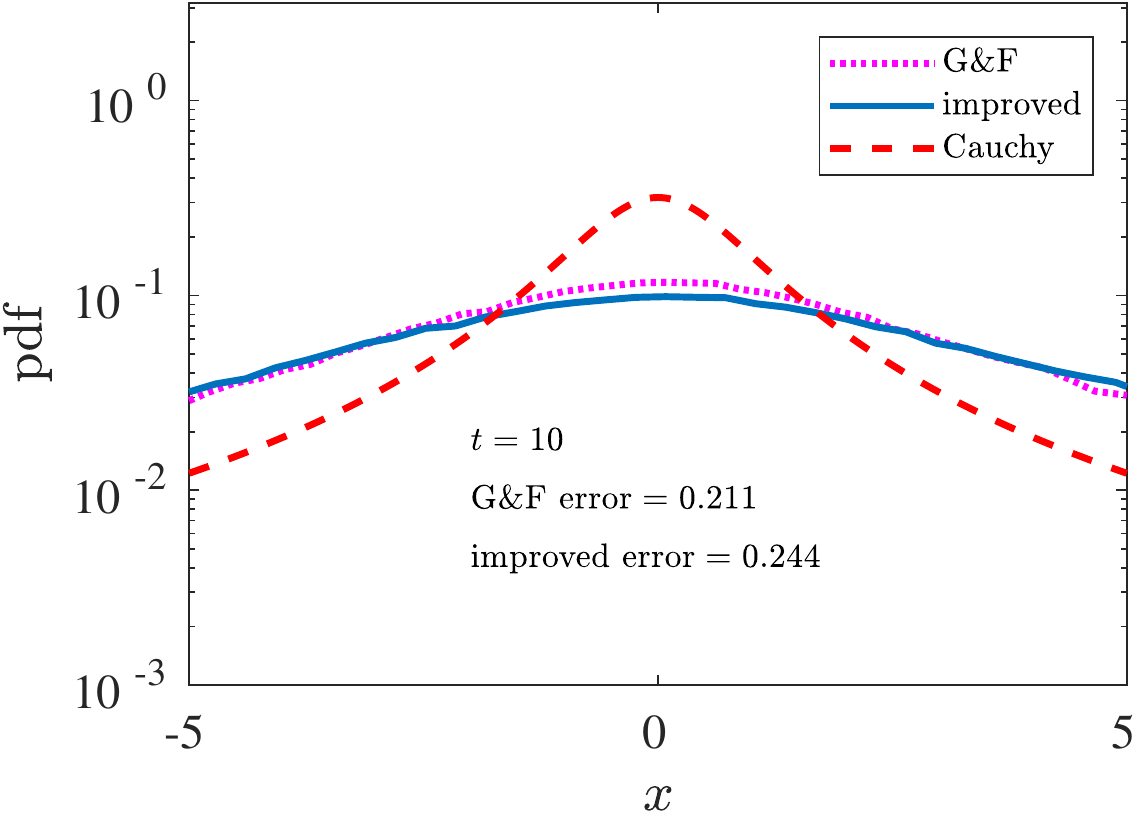}
  \hspace{.025\textwidth}
  \includegraphics[scale = .45]{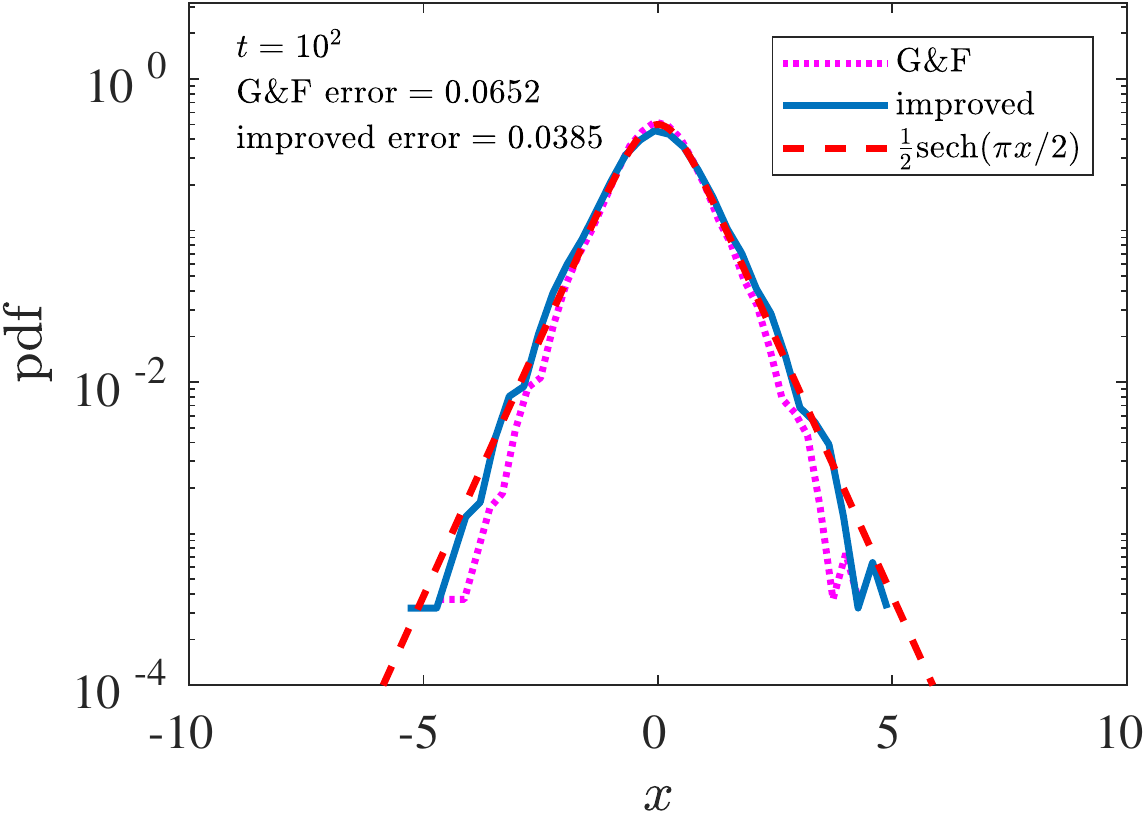}

  \includegraphics[scale = .45]{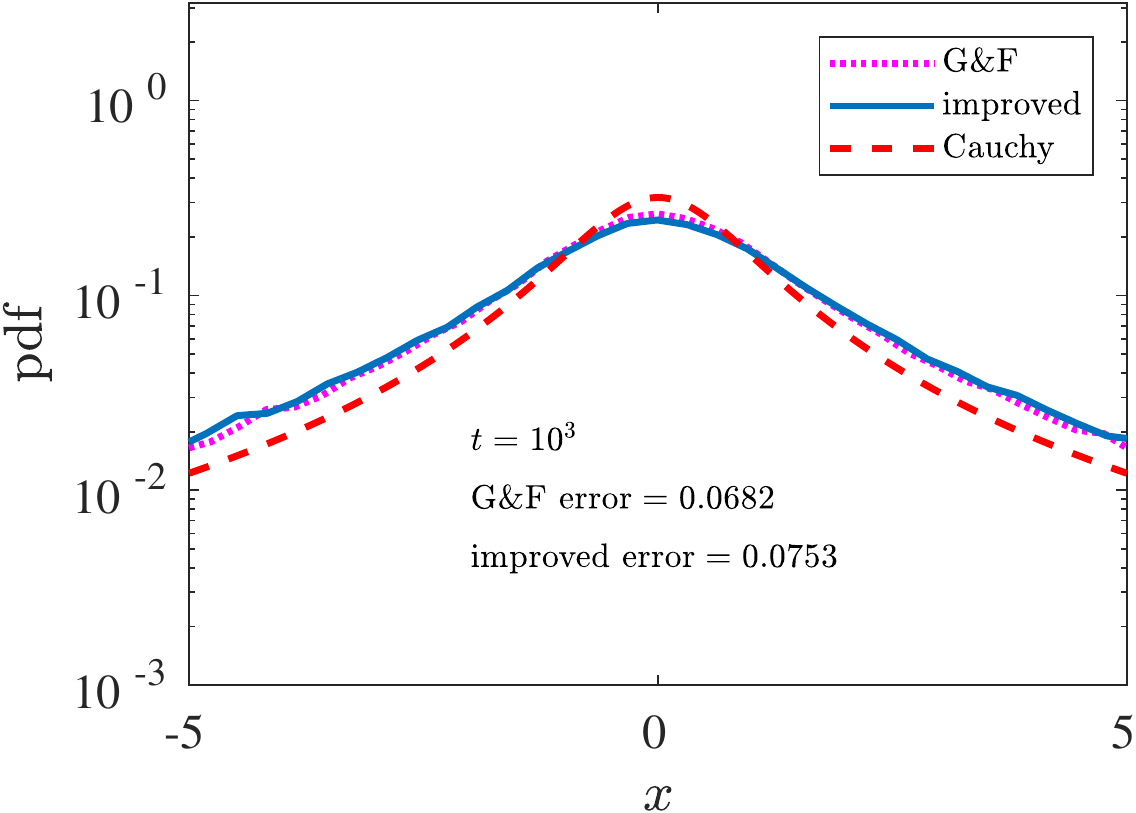}
  \hspace{.025\textwidth}
  \includegraphics[scale = .45]{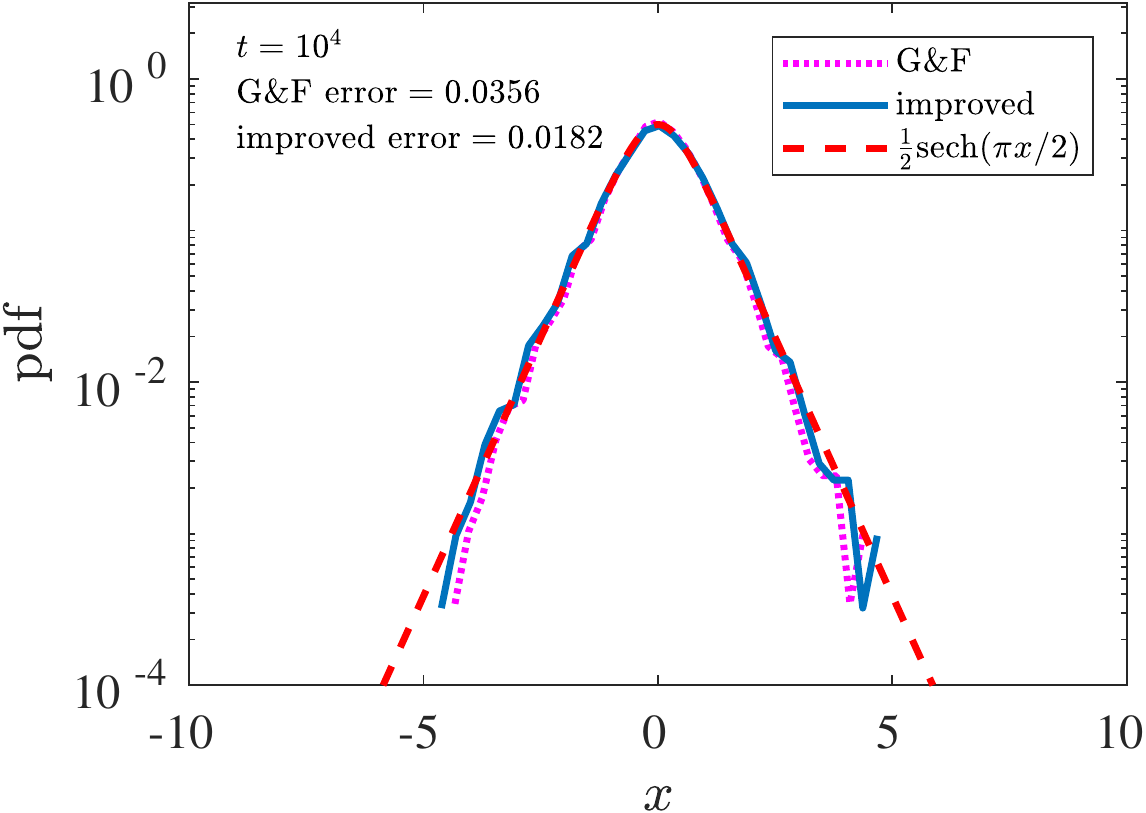}

  \includegraphics[scale = .45]{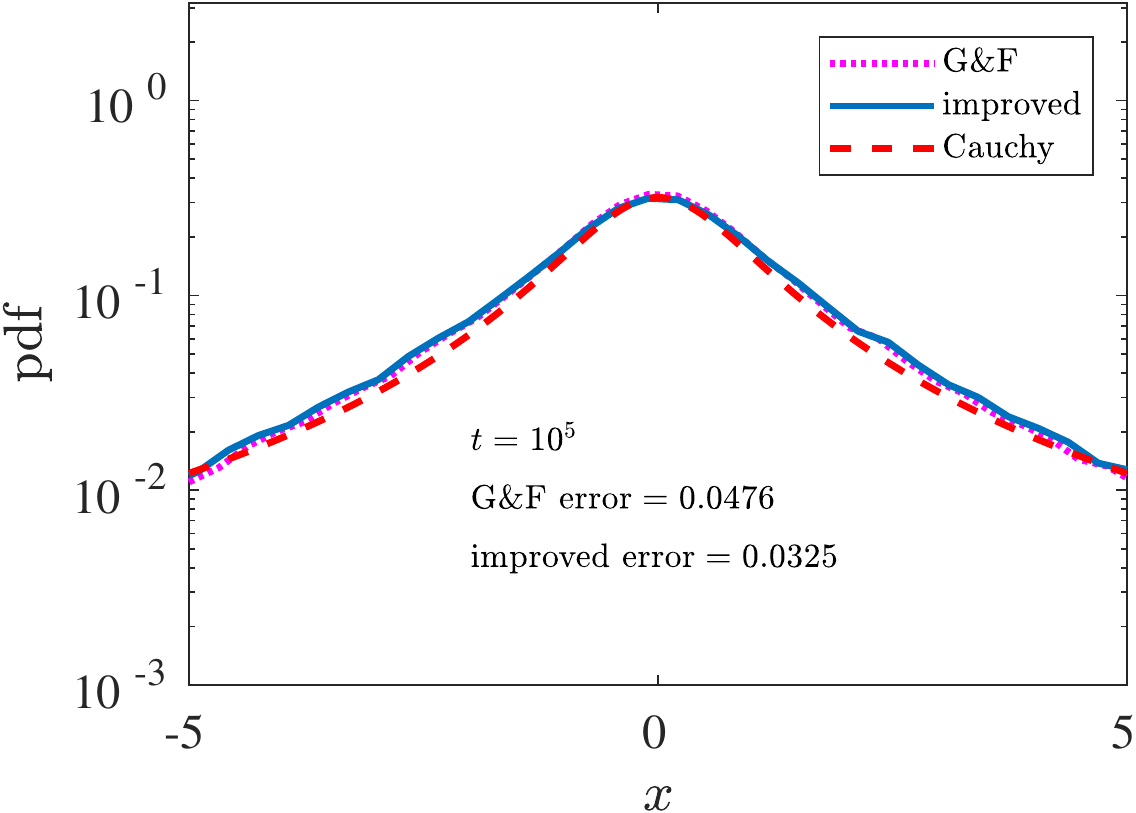}
  \hspace{.025\textwidth}
  \includegraphics[scale = .45]{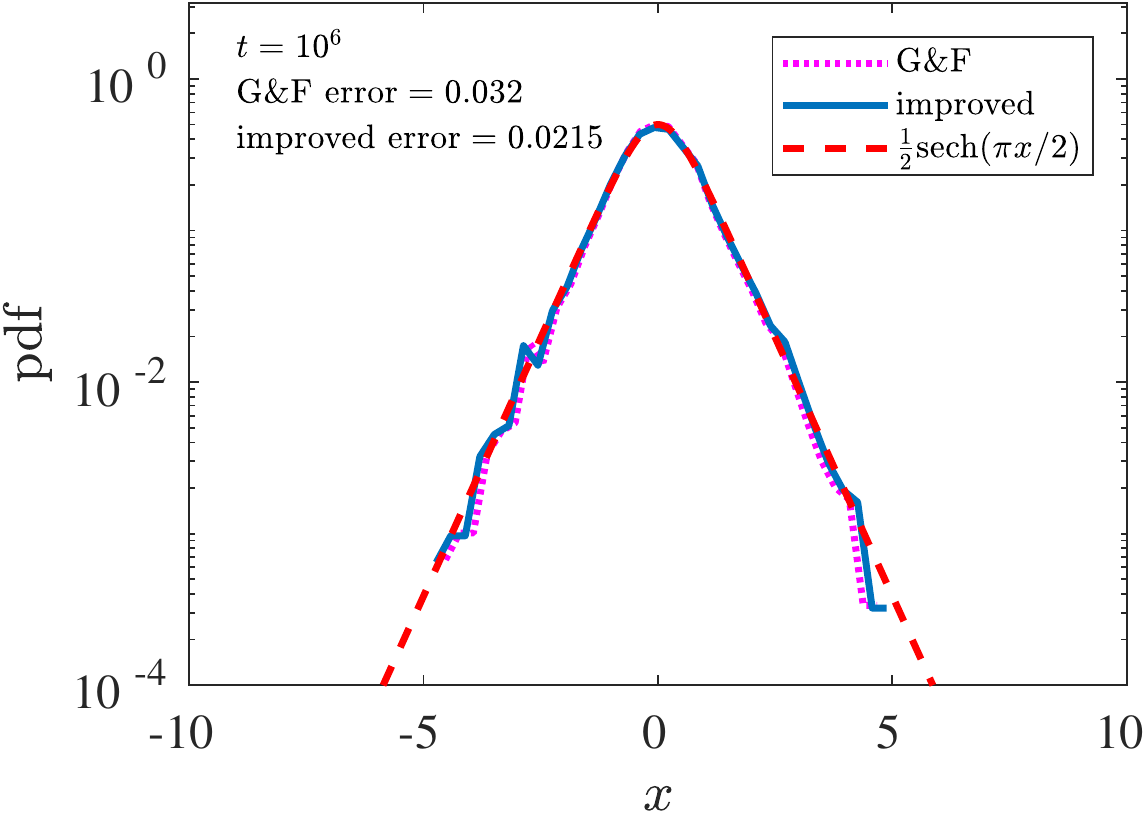}
  \caption{Numerical simulations without drift ($\beta=0$).  The left column
    shows~$10^5$ realizations for a point-like obstacle with $r_0=1$.  The
    right column is for~$10^4$ realizations of winding around a disk of radius
    $a=2$.  The error is the~$\Ltwo$ norm of the difference between the data
    and the limit distribution in each case.}
  \label{fig:correction}
\end{figure}

We compare the corrected
forms~\cref{eq:GFimprovement2,eq:GFsechimprovement} to a normalizer
without the factor~$\ee^\gamma$ for numerical simulations in
\cref{fig:correction}.  The improvement is marginal and can only be seen for
large times.  Nevertheless, we include it here for completeness and rigor.

\section*{Acknowledgments}

This research was supported by the US National Science Foundation, under grant
CMMI-1233935.

\bibliography{../winding.bib}

\end{document}